\documentclass[10pt]{amsart}

\usepackage{amssymb,amsmath,amsthm}
\usepackage[all]{xy}
\usepackage{eucal}

\usepackage{color}

\theoremstyle{plain}
\newtheorem{thm}[subsection]{Theorem}

\newtheorem{prop}[subsection]{Proposition}

\theoremstyle{definition}
\newtheorem{defn}[subsection]{Definition}

\theoremstyle{remark}
\newtheorem{rem}[subsection]{Remark}

\newcommand{\DD}{{ \mathsf{D} }}

\newcommand{\ZZ}{{ \mathbb{Z} }}
\newcommand{\NN}{{ \mathbb{N} }}

\newcommand{\capP}{{ \mathcal{P} }}

\newcommand{\capX}{{ \mathcal{X} }}
\newcommand{\capY}{{ \mathcal{Y} }}

\newcommand{\powerset} {{ \mathcal{P} }}

\newcommand{\Sk}{{ \mathrm{sk} }}

\newcommand{\ev}{{ \mathrm{ev} }}

\newcommand{\id}{{ \mathrm{id} }}

\newcommand{\Space}{{ \mathsf{S} }}

\newcommand{\Ho}{{ \mathsf{Ho} }}

\newcommand{\sSet}{{ \mathsf{sSet} }}

\newcommand{\Chaincx}{{ \mathsf{Ch} }}

\newcommand{\sAb}{{ \mathsf{sAb} }}

\newcommand{\M}{{ \mathsf{M} }}

\newcommand{\Set}{{ \mathsf{Set} }}

\newcommand{\RR}{{ \mathsf{R} }}

\newcommand{\K}{{ \mathsf{K} }}

\newcommand{\coAlg}{{ \mathsf{coAlg} }}
\newcommand{\res}{{ \mathsf{res} }}
\newcommand{\BK}{{ \mathsf{BK} }}
\newcommand{\CGHaus}{{ \mathsf{CGHaus} }}

\newcommand{\coAlgK}{{ \coAlg_\K }}

\newcommand{\Smash}{{ \,\wedge\, }}
\newcommand{\tensor}{{ \otimes }}

\newcommand{\tensordot}{{ \dot{\tensor} }}

\newcommand{\wequiv}{{ \ \simeq \ }}

\newcommand{\Iso}{{  \ \cong \ }}
\newcommand{\Equal}{{ \ = \ }}

\newcommand{\rarrow}{{ \rightarrow }}
\newcommand{\larrow}{{ \leftarrow }}

\newcommand{\function}[3]{{ {#1}\colon\thinspace{#2}\rarrow{#3} }}

\DeclareMathOperator*{\colim}{colim}
\DeclareMathOperator*{\holim}{holim}
\DeclareMathOperator{\Hombold}{\mathbf{Hom}}
\DeclareMathOperator{\hombold}{\mathbf{hom}}

\DeclareMathOperator{\Map}{Map}

\DeclareMathOperator{\Cobar}{Cobar}

\DeclareMathOperator{\Tot}{Tot}

\DeclareMathOperator{\hofib}{hofib}
\DeclareMathOperator{\iter}{iterated}

\DeclareMathOperator{\fiber}{fiber}

\begin{document}

\title[Integral chains]{Integral chains and Bousfield-Kan completion}

\author{Jacobson R. Blomquist}
\author{John E. Harper}

\address{Department of Mathematics, The Ohio State University, 231 West 18th Ave, Columbus, OH 43210, USA}
\email{blomquist.9@osu.edu}

\address{Department of Mathematics, The Ohio State University, Newark, 1179 University Dr, Newark, OH 43055, USA}
\email{harper.903@math.osu.edu}

\begin{abstract}
Working in the Arone-Ching framework for homotopical descent, it follows that the Bousfield-Kan completion map with respect to integral homology is the unit of a derived adjunction. We prove that this derived adjunction, comparing spaces with coalgebra complexes over the associated integral homology comonad, via integral chains, can be turned into a derived equivalence by replacing spaces with the full subcategory of simply connected spaces. In particular, this provides an integral chains characterization of the homotopy type of simply connected spaces.
\end{abstract}

\maketitle

\section{Introduction}

In this paper we revisit Bousfield-Kan \cite{Bousfield_Kan} completion of spaces with respect to integral homology. Our aim is to clarify and conceptualize the important completion result proved in \cite{Bousfield_Kan} that homology completion for simply connected spaces recovers the original space (up to homotopy)---our main result, Theorem \ref{MainTheorem}, recasts the Bousfield-Kan completion theorem in terms of coalgebraic structures encoding simply connected homotopy types.

Bousfield-Kan completion is constructed by gluing together (via a homotopy limit) a cosimplicial resolution that is nothing more than iterations of homology. The reason these homology invariants are useful is that they throw away information about the space, thus making computations of these invariants ``easier'', at the cost of losing information. The classical completion result of Bousfield-Kan indicates that if fundamental group information ``takes care of itself'', then no information is really lost, it somehow gets preserved or encoded without loss in the coface maps and codegeneracy maps of the resolution. In other words, in this paper we ask the deeper question: What is the completion result of Bousfield-Kan really telling us? Unwinding what the cosimplicial identities mean reveals that they are encoding nothing other than the fact that homology is equipped with a coaction of the homology comonad up to all higher coherences: i.e., 3-fold coassociativity, 4-fold coassociativity, 5-fold coassociativity, etc., mixed in with various co-unit maps and diagrams (see \cite[VII.2]{MacLane_categories} for a useful discussion of coherence diagrams). This suggests the following idea: Maybe the homotopy category of 1-connected simplicial abelian groups equipped with a coaction of the homology comonad captures the entire homotopy category of simply connected spaces---that is our main result. In more detail:

In this paper we investigate the homotopy category of 1-connected simplicial abelian groups equipped with the coalgebraic structure naturally arising on the space level homology $\tilde{\ZZ}X$ (Section \ref{sec:integral_chains}) of a pointed space $X$. This coalgebraic structure is encoded by the coaction of the associated homology comonad $\K=\tilde{\ZZ}U$ whose underlying functor on simplicial abelian groups assigns to $Y$ the reduced free abelian group complex generated by the underlying simplicial set $UY$; the comultiplication map $\K\rarrow\K\K$ is induced from the canonical unit map $\id\rightarrow U\tilde{\ZZ}$, which sends an element $x$ to the element underlying the (equivalence class of the) formal sum $1\cdot x$; i.e., comultiplication is the canonical map $\tilde{\ZZ}\id U\rarrow \tilde{\ZZ}U\tilde{\ZZ}U$. We show that this homotopy category of 1-connected coalgebras, built by appropriately fattening up ``all constructions in sight'' (to be homotopy meaningful) via the Arone-Ching enrichments \cite[Section 1]{Arone_Ching_classification}, is equivalent to the homotopy category of simply connected spaces.

Our hope is that the characterization of simply connected homotopy types described in this paper will provide a jumping off point for future work in using ``homology equipped with coalgebraic structure'' to distinguish homotopy types of simply connected
spaces. The informal slogan is: Naturally occurring coalgebraic
structures on integral chains, while less traditional to exploit, are
potentially more powerful and sensitive than, say, dualizing in order to work with algebraic objects. With an eye on future calculational outcomes of this work, the next step in this project is the construction of concrete chain-level descriptions of the associated homology comonad $\K=\tilde{\ZZ}U$ and its coalgebras; i.e., a derived-equivalent category of coalgebras on chain complexes. A natural candidate is provided via the normalization-denormalization comparison (in the Dold-Kan theorem) to get down to the level of chains (where homological algebra becomes available); working out the precise analysis and details is beyond the scope of our work here.

This paper is written simplicially so that ``space'' means ``simplicial set'' unless otherwise noted; see Bousfield-Kan \cite[VIII]{Bousfield_Kan} and Goerss-Jardine \cite[I]{Goerss_Jardine}.

\subsection{The space level Hurewicz map}

If $X$ is a pointed space, the usual Hurewicz map between homotopy groups and reduced homology groups has the form
\begin{align}
\label{eq:hurewicz_map_introduction}
  \pi_*(X)\rarrow \tilde{H}_*(X;\ZZ).
\end{align}
The starting point of the work in Bousfield-Kan \cite[I.2.3]{Bousfield_Kan} is essentially  the observation that this comparison map comes from a space level Hurewicz map of the form
\begin{align}
\label{eq:hurewicz_map_spaces_level_introduction}
  X\rightarrow \tilde{\ZZ}(X)
\end{align}
and that applying $\pi_*$ recovers the map \eqref{eq:hurewicz_map_introduction}.

\subsection{Iterating the Hurewicz map to build a resolution}

Once one has such a Hurewicz map on the level of spaces, it is natural to form a cosimplicial resolution of $X$ with respect to integral homology of the form
\begin{align}
\label{eq:homology_resolution_introduction}
\xymatrix@1{
  X\ar[r] &
  \tilde{\ZZ}(X)\ar@<-0.5ex>[r]\ar@<0.5ex>[r] &
  \tilde{\ZZ}^2(X)
  \ar@<-1.0ex>[r]\ar[r]\ar@<1.0ex>[r] &
  \tilde{\ZZ}^3(X)\cdots
  }
\end{align}
showing only the coface maps. The codegeneracy maps, not shown above, are induced by the counit $\K\rarrow\id$ of the associated integral homology comonad $\K$ that can be thought of as encoding the space level co-operations on the integral homology complexes; compare with Miller \cite[Section 1]{Miller} in the context of the Sullivan conjecture. 

In other words, by iterating the space level Hurewicz map \eqref{eq:hurewicz_map_spaces_level_introduction}, Bousfield-Kan \cite[I.4.1]{Bousfield_Kan} build a cosimplicial resolution of $X$ with respect to integral homology, and taking the homotopy limit of the resolution \eqref{eq:homology_resolution_introduction} produces the $\ZZ$-completion map
\begin{align}
\label{eq:comparison_map_Z_completion}
  X\rarrow X^\wedge_\ZZ
\end{align}
which can be interpreted in the case of simply connected or nilpotent spaces as the localization of $X$ with respect to integral homology; see  Bousfield \cite[Section 1]{Bousfield_localization_spaces} for this, together with the associated universal property characterization of localization behind the argument, and Bousfield-Kan \cite[V.4]{Bousfield_Kan}. This is the integral analog of the completions and localizations of spaces originally studied in Sullivan \cite{Sullivan_mit_notes, Sullivan_genetics}, and subsequently in Bousfield-Kan \cite{Bousfield_Kan} and Hilton-Mislin-Roitberg \cite{Hilton_Mislin_Roitberg}; there is an extensive literature---for a useful introduction, see also Bousfield \cite{Bousfield_localization_spaces}, Dwyer \cite{Dwyer_localizations}, and May-Ponto \cite{May_Ponto}; a transfinite version of ``iterating the Hurewicz map'' is studied in Dror-Dwyer \cite{Dror_Dwyer_long_homology}.

\subsection{The main result}

Working in the Arone-Ching \cite{Arone_Ching_classification} framework for homotopical descent, it follows that the Bousfield-Kan completion map with respect to integral homology is the unit of a derived adjunction \eqref{eq:derived_adjunction_main}. Our main result is the following theorem (Theorem \ref{MainTheorem}): that this derived adjunction can be turned into a derived equivalence by replacing spaces with the full subcategory of simply connected spaces. This is reminiscent of Quillen's \cite{Quillen_rational} rational chains equivalence in rational homotopy theory and, dually,  Sullivan's \cite{Sullivan_infinitesimal} rational cochains equivalence; compare Bousfield-Gugenheim \cite{Bousfield_Gugenheim} for an adjoint functor approach to Sullivan's theory using methods of Quillen's model categories. Since the foundational work of Quillen and Sullivan, the problem of establishing $p$-adic and fully integral versions of the rational chains equivalence theorem, or its dual, have been studied in the work of Dwyer-Hopkins \cite{Dwyer_Hopkins}, Goerss \cite{Goerss_simplicial_chains}, Karoubi \cite{Karoubi}, Kriz \cite{Kriz}, Mandell \cite{Mandell, Mandell_cochains_homotopy_type}, and Smirnov \cite{Smirnov}; we have been motivated and inspired by their work. 

The Arone-Ching framework \cite{Arone_Ching_classification} constructs a highly homotopy coherent topological enrichment \cite[1.14]{Arone_Ching_classification} for $\K$-coalgebras comprising the collections of maps called derived $\K$-coalgebra maps \cite[1.11]{Arone_Ching_classification}---the underlying maps are required to respect the $\K$-coalgebra structure (i.e., the $\K$-coaction), but only in a highly homotopy coherent manner. In this framework, $\Map_\coAlgK(Y,Y')$ denotes the space of all derived $\K$-coalgebra maps from $Y$ to $Y'$. This highly homotopy coherent topological enrichment on $\K$-coalgebras \cite[Section 1]{Arone_Ching_classification} (see Sections \ref{sec:homotopy_theory_K_coalgebras} and \ref{sec:derived_fundamental_adjunction} for a brief development in the context of this paper) provides a framework that allows one to analyze homotopical descent without being forced into a direct analysis of limits in the category of $\K$-coalgebras and strict $\K$-coalgebra maps: that is the main payoff of the Arone-Ching enrichments.

In terms of these enrichments, the upshot of our main result (Theorem \ref{MainTheorem}) is that the integral chains functor $\tilde{\ZZ}$ in \eqref{eq:derived_adjunction_main} is a Dywer-Kan equivalence between 1-connected spaces, equipped with their usual topological enrichment, and 1-connected $\K$-coalgebras, equipped with their highly homotopy coherent topological enrichment built in Arone-Ching \cite[Section 1]{Arone_Ching_classification}. In the statement of the following theorem, $C(Y)$ is the usual cosimplicial cobar construction (Definition \ref{defn:cobar_construction}) associated to the $\K$-coalgebra $Y$.

\begin{thm}
\label{MainTheorem}
The integral chains functor $\tilde{\ZZ}$ fits into a derived adjunction
\begin{align}
\label{eq:derived_adjunction_main}
  \Map_{\coAlgK}(\tilde{\ZZ}X,Y)\wequiv
  \Map_{\Space_*}(X,\holim\nolimits_\Delta C(Y))
\end{align}
comparing pointed spaces to coalgebra complexes over the associated integral homology comonad $\K=\tilde{\ZZ}U$, that is a Dwyer-Kan equivalence after restriction to the full subcategories of $1$-connected spaces and $1$-connected cofibrant $\K$-coalgebras, with respect to the enrichments discussed above (see Remark \ref{rem:elaboration_on_the_main_result})
\end{thm}

\begin{rem}
\label{rem:elaboration_on_the_main_result}
The proof boils down to the following assertions on the integral chains functor $\tilde{\ZZ}$:
\begin{itemize}
\item[(a)] If $Y$ is a $1$-connected cofibrant $\K$-coalgebra, then the derived counit map
\begin{align*}
  \tilde{\ZZ}\holim\nolimits_\Delta C(Y)\xrightarrow{\wequiv} Y
\end{align*}
associated to \eqref{eq:derived_adjunction_main} is a weak equivalence.
\item[(b)] If $X'$ is a $1$-connected space, then the derived unit map
\begin{align*}
  X'\xrightarrow{\wequiv}\holim\nolimits_\Delta 
  C(\tilde{\ZZ}X')
\end{align*} 
associated to \eqref{eq:derived_adjunction_main} is tautologically the Bousfield-Kan $\ZZ$-completion map $X'\rarrow {X'}^\wedge_\ZZ$, and hence is a weak equivalence by \cite[III.5.4]{Bousfield_Kan}; in particular, the integral chains functor induces a weak equivalence
\begin{align}
  \label{eq:homotopically_fully_faithful_chains}
  \tilde{\ZZ}\colon\thinspace\Map^h_{\Space_*}(X,X')\wequiv\Map_\coAlgK(\tilde{\ZZ}X,\tilde{\ZZ}X')
\end{align}
on mapping spaces and hence is homotopically fully faithful on $1$-connected spaces.
\end{itemize}
Here, the realization of the Dwyer-Kan \cite[4.7]{Dwyer_Kan_function_complexes} homotopy function complex, denoted $\Map^h_{\Space_*}(X,X')$, can be replaced with the realization of the usual mapping complex, denoted $\Map_{\Space_*}(X,X')$, if $X'$ is fibrant in $\Space_*$. The right-hand side of \eqref{eq:homotopically_fully_faithful_chains} denotes the space of all derived $\K$-coalgebra maps from $\tilde{\ZZ}X$ to $\tilde{\ZZ}X'$ \cite[1.10]{Arone_Ching_classification} (Definition \ref{defn:mapping_spaces_of_derived_K_coalgebras}).
\end{rem}

\subsection{Corollaries of the main result}
The following are corollaries of the main result (Theorem \ref{MainTheorem}). Integral cochains versions of the first two results below were previously established by Mandell \cite[0.1]{Mandell_cochains_homotopy_type}, assuming additional finite type conditions; an interesting approach to some of the integral cochains results in \cite{Mandell_cochains_homotopy_type} was subsequently developed in Karoubi \cite{Karoubi}. In \cite[0.1]{Mandell_cochains_homotopy_type} it is shown that the integral cochains functor cannot be full on the homotopy category; on the other hand, an  advantage of the cochains setup is that $E_\infty$ cochain algebras have a more familiar ``homological algebra'' feel to them than, say, $\K$-coalgebras.

\begin{thm}[Classification theorem]
\label{thm:classification}
A pair of $1$-connected pointed spaces $X$ and $X'$ are weakly equivalent if and only if the integral chains $\tilde{\ZZ}X$ and $\tilde{\ZZ}X'$ are weakly equivalent as derived $\K$-coalgebras.
\end{thm}

\begin{thm}[Classification of maps theorem]
\label{thm:classification_maps}
Let $X,X'$ be pointed spaces. Assume that $X'$ is $1$-connected and fibrant.
\begin{itemize}
\item[(a)] \emph{(Existence)} Given any map $\phi$ in $[\tilde{\ZZ}X,\tilde{\ZZ}X']_\K$, there exists a map $f$ in $[X,X']$ such that $\phi=\tilde{\ZZ}(f)$.
\item[(b)] \emph{(Uniqueness)} For each pair of maps $f,g$ in $[X,X']$, $f=g$ if and only if $\tilde{\ZZ}(f)=\tilde{\ZZ}(g)$ in the homotopy category of $\K$-coalgebras.
\end{itemize}
\end{thm}

\begin{thm}[Characterization theorem]
\label{thm:characterization}
A cofibrant $\K$-coalgebra $Y$ is weakly equivalent, via derived $\K$-coalgebra maps, to the integral chains $\tilde{\ZZ}X$ of some $1$-connected space $X$ if and only if $Y$ is $1$-connected.
\end{thm}

\subsection{Strategy of attack}

We were encouraged by the results in \cite{Goerss_simplicial_chains} showing that the Bousfield-Kan completion map should be the derived unit map of a derived comparison adjunction between spaces and coalgebraic data; this result foreshadows the later developments and ideas in \cite{Arone_Ching_classification, Francis_Gaitsgory, Hess} on homotopical descent. Our argument, motivated by \cite{Ching_Harper_derived_Koszul_duality}, involves leveraging Goodwillie's higher dual Blakers-Massey theorem \cite[2.6]{Goodwillie_calculus_2}, together with the ``uniformity of faces'' behavior forced by the cosimplicial identities via existence of appropriate retractions, along with strong estimates for the uniform cartesian-ness of iterations of the Hurewicz map provided by Dundas' higher Hurewicz theorem \cite[2.6]{Dundas_relative_K_theory}, to obtain connectivity estimates for commuting the left derived integral chains functor past the right derived limit of the associated cosimplicial cobar construction on coalgebra complexes.

\subsection{Commuting integral chains with holim of a cobar construction}
Once the framework is setup, the main result boils down to proving that the left derived integral chains functor $\tilde{\ZZ}$ commutes,\begin{align}
\label{eq:key_technical_result}
  \tilde{\ZZ}\holim\nolimits_\Delta C(Y)\wequiv
  \holim\nolimits_\Delta \tilde{\ZZ}C(Y)
\end{align}
up to weak equivalence, with the right derived limit functor $\holim_\Delta $, when composed with the cosimplicial cobar construction $C$ associated to integral homology and evaluated on $1$-connected coalgebra complexes over $\K$; but our homotopical estimates are stronger---they prove strong convergence of the associated homotopy spectral sequence (Theorem \ref{thm:strong_convergence_ss}).

\subsection{Organization of the paper}

In Section \ref{sec:outline_of_the_argument} we outline the argument of our main result. The proof naturally breaks up into four subsidiary results. In Section \ref{sec:proofs} we review the integral chains functor and then prove the main result. Sections \ref{sec:simplicial_structures} and \ref{sec:homotopy_limit_towers} are background sections; for the convenience of the reader we briefly recall some preliminaries on simplicial structures and homotopy limits that are essential to understanding this paper. In Section \ref{sec:homotopy_theory_K_coalgebras} we recall briefly the Arone-Ching enrichments and associated homotopy theory of $\K$-coalgebras in the context of this paper, and in Section \ref{sec:derived_fundamental_adjunction} the associated derived adjunction. For the experts, who are also familiar with the enrichments in Arone-Ching \cite{Arone_Ching_classification}, it should suffice to read Sections \ref{sec:outline_of_the_argument} and \ref{sec:proofs} for a complete proof of the main result.

\subsection*{Acknowledgments}

The authors would like to thank to Michael Ching for useful discussions throughout this project.  The second author would like to thank Bj\o rn Dundas, Bill Dwyer, Haynes Miller, and Crichton Ogle for useful remarks, and Mark Behrens, Dan Burghelea, Lars Hesselholt, Rick Jardine, Mike Mandell, Nath Rao and Dennis Sullivan for helpful comments. The second author is grateful to Haynes Miller for a stimulating and enjoyable visit to the Massachusetts Institute of Technology in early spring 2015, and to Bj\o rn Dundas for a stimulating and enjoyable visit to the University of Bergen in late spring 2015, and for their invitations which made this possible. The authors would like to thank an anonymous referee for helpful comments and suggestions. The first author was supported in part by National Science Foundation grants DMS-1510640 and  DMS-1547357.

\section{Outline of the argument}
\label{sec:outline_of_the_argument}

We will now outline the proof of our main result (Theorem \ref{MainTheorem}). Since the derived unit map is tautologically the Bousfield-Kan $\ZZ$-completion map $X'\rarrow {X'}^\wedge_\ZZ$, which is proved to be a weak equivalence on $1$-connected spaces in Bousfield-Kan \cite[III.5.4]{Bousfield_Kan}, proving the main result reduces to verifying that the derived counit map is a weak equivalence. Our attack strategy naturally breaks up into four subsidiary results; Theorems \ref{thm:estimating_connectivity_of_maps_in_tower_C_of_Y}, \ref{thm:connectivities_for_map_into_n_th_stage}, \ref{thm:connectivities_for_map_that_commutes_Z_into_inside_of_holim}, and \ref{thm:ZC_commutes_with_desired_holim}.

The following theorem is proved in Section \ref{sec:proofs} (just after Proposition \ref{prop:iterated_hofiber_codegeneracy_cube}).

\begin{thm}
\label{thm:estimating_connectivity_of_maps_in_tower_C_of_Y}
If $Y$ is a $1$-connected cofibrant $\K$-coalgebra and $n\geq 1$, then the natural map
\begin{align}
\label{eq:tower_map_from_n_th_stage_to_next_lower_stage}
  \holim\nolimits_{\Delta^{\leq n}}C(Y)&\rightarrow
  \holim\nolimits_{\Delta^{\leq n-1}}C(Y)
\end{align}
is an $(n+2)$-connected map between $1$-connected objects.
\end{thm}

\begin{thm}
\label{thm:connectivities_for_map_into_n_th_stage}
If $Y$ is a $1$-connected cofibrant $\K$-coalgebra and $n\geq 0$, then the natural maps
\begin{align}
\label{eq:canonical_map_needed_to_discuss}
  \holim\nolimits_\Delta C(Y)&\rightarrow
  \holim\nolimits_{\Delta^{\leq n}}C(Y)\\
\label{eq:the_other_canonical_map_needed}
  \tilde{\ZZ}\holim\nolimits_\Delta C(Y)&\rightarrow
  \tilde{\ZZ}\holim\nolimits_{\Delta^{\leq n}}C(Y)
\end{align}
are $(n+3)$-connected maps between $1$-connected objects.
\end{thm}

\begin{proof}
Consider the first part. By Theorem \ref{thm:estimating_connectivity_of_maps_in_tower_C_of_Y} each of the maps in the holim tower $\{\holim_{\Delta^{\leq n}}C(Y)\}_n$, above level $n$, is at least $(n+3)$-connected. It follows that the map \eqref{eq:canonical_map_needed_to_discuss} is $(n+3)$-connected. The second part follows from the first part, since by the Hurewicz theorem the integral chains functor $\tilde{\ZZ}$ preserves such connectivities.
\end{proof}

\begin{rem}
It is worth pointing out that $\holim_\Delta$ (resp. $\holim_{\Delta^{\leq n}}$) (Definition \ref{defn:homotopy_limit_derived}) is a derived version of the familiar $\Tot$ (resp. $\Tot_n$) (Definitions \ref{defn:totalization_and_restricted_totalization} and \ref{defn:totalization_functors}, and Proposition \ref{prop:comparing_holim_with_Tot}).
\end{rem}

We prove the following theorem in Section \ref{sec:proofs} (following Theorem \ref{thm:cocartesian_and_cartesian_estimates}). At the technical heart of the proof lies Goodwillie's higher dual Blakers-Massey theorem \cite[2.6]{Goodwillie_calculus_2} (Proposition \ref{prop:higher_dual_blakers_massey}). To carry out this line of attack, the input to \cite[2.6]{Goodwillie_calculus_2} requires the homotopical analysis of an $\infty$-cartesian $(n+1)$-cube associated to the $n$-th stage, $\holim_{\Delta^{\leq n}}C(Y)$, of the $\holim$ tower associated to $C(Y)$; it is built from coface maps in $C(Y)$ (Definition \ref{defn:the_wide_tilde_construction}). The needed homotopical analysis is worked out by leveraging the strong uniform cartesian-ness estimates for iterations of the Hurewicz map, applied to $X=UY$, in Dundas' higher Hurewicz theorem \cite[2.6]{Dundas_relative_K_theory} (Proposition \ref{prop:higher_hurewicz_theorem}), together with the ``uniformity of faces'' behavior forced by the cosimplicial identities (see \eqref{eq:sequence_paths_1_cubes}, \eqref{eq:sequence_paths_of_squares}, and Proposition \ref{prop:comparing_faces_of_coface_cube_with_codegeneracy_cube}) which ensures that the ``other faces'' (any of the ones involving the $\K$-coaction map on $Y$) required for input to \cite[2.6]{Goodwillie_calculus_2} have similar cartesian-ness estimates forced on them.

\begin{thm}
\label{thm:connectivities_for_map_that_commutes_Z_into_inside_of_holim}
If $Y$ is a $1$-connected cofibrant $\K$-coalgebra and $n\geq 1$, then the natural map
\begin{align}
\label{eq:commuting_Z_past_holim_delta}
  \tilde{\ZZ}\holim\nolimits_{\Delta^{\leq n}} C(Y)\rightarrow
  \holim\nolimits_{\Delta^{\leq n}} \tilde{\ZZ}\,C(Y),
\end{align}
is $(n+5)$-connected; the map is a weak equivalence for $n=0$.
\end{thm}

The following is a corollary of these connectivity estimates, together with a left cofinality argument in \cite[3.16]{Dror_Dwyer_long_homology}.

\begin{thm}
\label{thm:ZC_commutes_with_desired_holim}
If $Y$ is a $1$-connected cofibrant $\K$-coalgebra, then the natural maps
\begin{align}
\label{eq:ZC_commutes_with_desired_holim}
  \tilde{\ZZ}\holim\nolimits_\Delta C(Y)&\xrightarrow{\wequiv}
  \holim\nolimits_\Delta \tilde{\ZZ}\,C(Y)\xrightarrow{\wequiv}Y
\end{align}
are weak equivalences.
\end{thm}

\begin{proof}
Consider the left-hand map. It suffices to verify that the connectivities of the natural maps \eqref{eq:the_other_canonical_map_needed} and \eqref{eq:commuting_Z_past_holim_delta}
 are strictly increasing with $n$, and Theorems \ref{thm:connectivities_for_map_into_n_th_stage} and \ref{thm:connectivities_for_map_that_commutes_Z_into_inside_of_holim} complete the proof. Consider the case of the right-hand map. Since the cosimplicial cobar construction $\Cobar(\K,\K,Y)$, which is isomorphic to $\tilde{\ZZ}\, C(Y)$, has extra codegeneracy maps $s^{-1}$ (\cite[6.2]{Dwyer_Miller_Neisendorfer}), it follows from the cofinality argument in  \cite[3.16]{Dror_Dwyer_long_homology}, together with the fact that
$
  Y\xrightarrow{d^0}\Cobar(\K,\K,Y)\xrightarrow{s^{-1}} Y
$
factors the identity, that the right-hand map in \eqref{eq:ZC_commutes_with_desired_holim} is a weak equivalence.
\end{proof}

\begin{proof}[Proof of Theorem \ref{MainTheorem}]
We want to verify that the natural map
$
  \tilde{\ZZ}\holim\nolimits_\Delta C(Y)\xrightarrow{\wequiv} Y
$
is a weak equivalence; since this is the composite \eqref{eq:ZC_commutes_with_desired_holim}, Theorem \ref{thm:ZC_commutes_with_desired_holim} completes the proof; this reduction argument can be thought of as a homotopical Barr-Beck comonadicity theorem (see \cite[2.20]{Arone_Ching_classification}).
\end{proof}

\section{Homotopical analysis}
\label{sec:proofs}

The purpose of this section is to prove Theorems \ref{thm:estimating_connectivity_of_maps_in_tower_C_of_Y} and \ref{thm:connectivities_for_map_that_commutes_Z_into_inside_of_holim}. 

\subsection{Integral chains}
\label{sec:integral_chains}

The functor $\tilde{\ZZ}$ is equipped with a coaction over the comonad $\K$ that appears in the Bousfield-Kan $\ZZ$-completion construction; this observation, which remains true for any adjunction provided that the indicated limits below exist, forms the basis of the homotopical descent ideas appearing in \cite{Arone_Ching_classification, Francis_Gaitsgory, Hess}.

Consider any pointed space $X$ and recall that $\tilde{\ZZ}(X):=\ZZ(X)/\ZZ(*)$. Then there is an adjunction
\begin{align}
\label{eq:homology_adjunction}
\xymatrix{
  \Space_*\ar@<0.5ex>[r]^-{\tilde{\ZZ}} &
  \sAb\ar@<0.5ex>[l]^-{U}
}
\end{align}
with left adjoint on top and $U$ the forgetful functor. Associated to the adjunction in \eqref{eq:homology_adjunction} is the monad $U\tilde{\ZZ}$ on pointed spaces $\Space_*$ and the comonad $\K:=\tilde{\ZZ}U$ on simplicial abelian groups $\sAb$ of the form
\begin{align}
\label{eq:TQ_homology_spectrum_functor_natural_transformations}
  \id\xrightarrow{\eta} U\tilde{\ZZ}&\quad\text{(unit)},\quad\quad\quad
  &\id\xleftarrow{\varepsilon}\K& \quad\text{(counit)}, \\
  \notag
  U\tilde{\ZZ}U\tilde{\ZZ}\rarrow U\tilde{\ZZ}&\quad\text{(multiplication)},\quad\quad\quad
  &\K\K\xleftarrow{m}\K& \quad\text{(comultiplication)}.
\end{align}
and it follows formally that there is a factorization of adjunctions of the form
\begin{align}
\label{eq:factorization_of_adjunctions_Z_homology}
\xymatrix{
  \Space_*\ar@<0.5ex>[r]^-{\tilde{\ZZ}} &
  \coAlgK\ar@<0.5ex>[r]\ar@<0.5ex>[l]^-{\lim_\Delta C} &
  \sAb\ar@<0.5ex>[l]^-{\K} 
}
\end{align}
with left adjoints on top and $\coAlgK\rarrow\sAb$ the forgetful functor; here, $\coAlgK$ denotes the category of $\K$-coalgebras and their morphisms (just after Remark \eqref{rem:initial_object_K_coalgebras} and \cite[1.2]{Arone_Ching_classification}). In particular, the integral homology complex $\tilde{\ZZ}X$ is naturally equipped with a $\K$-coalgebra structure. While we defer the definition of $C$ to the next subsection (Definition \ref{defn:cobar_construction}), to understand the comparison in \eqref{eq:factorization_of_adjunctions_Z_homology} between $\Space_*$ and $\coAlgK$ it suffices to know that $\lim_\Delta C(Y)$ is naturally isomorphic to an equalizer of the form
\begin{align*}
  \lim\nolimits_\Delta C(Y)\Iso
  \lim\Bigl(
  \xymatrix{
    UY\ar@<0.5ex>[r]^-{d^0}\ar@<-0.5ex>[r]_-{d^1} &
    U\K Y
  }
  \Bigr)
\end{align*}
where $d^0=m\id$, $d^1=\id m$, $\function{m}{U}{U\K=U\tilde{\ZZ}U}$ denotes the $\K$-coaction map on $U$ (defined by $m:=\eta\id$), and $\function{m}{Y}{\K Y}$ denotes the $\K$-coaction map on $Y$; this is because of the following property of cosimplicial objects (see Definition \ref{defn:cobar_construction}).

\begin{prop}
\label{prop:lim_of_cosimplicial_object}
Let $\M$ be a category with all small limits. If $A\in\M^\Delta$ (resp. $B\in\M^{\Delta_\res}$), then its limit is naturally isomorphic to an equalizer of the form
\begin{align*}
  \lim\nolimits_{\Delta}A\Iso
  \lim\bigl(
  \xymatrix@1{
    A^0\ar@<0.5ex>[r]^-{d^0}\ar@<-0.5ex>[r]_-{d^1} &
    A^1
  }
  \bigr)\quad\quad
  \Bigl(
  \text{resp.}\quad
  \lim\nolimits_{\Delta_\res}B\Iso
  \lim\bigl(
  \xymatrix@1{
    B^0\ar@<0.5ex>[r]^-{d^0}\ar@<-0.5ex>[r]_-{d^1} &
    B^1
  }
  \bigr)
  \Bigr)
\end{align*}
in $\M$, with $d^0$ and $d^1$ the indicated coface maps of $A$ (resp. $B$).
\end{prop}

\begin{proof}
This follows easily by using the cosimplicial identities \cite[I.1]{Goerss_Jardine} to verify the universal property of limits.
\end{proof}

\subsection{The cosimplicial cobar construction}

It will be useful to interpret the cosimplicial integral homology resolution of $X$ in terms of the following cosimplicial cobar construction involving the comonad $\K$ on $\sAb$. First note that associated to the adjunction $(\tilde{\ZZ},U)$ is a right $\K$-coaction $\function{m}{U}{U\K}$ on $U$ (defined by $m:=\eta\id$) and a left $\K$-coaction (or $\K$-coalgebra structure) $\function{m}{\tilde{\ZZ}X}{\K \tilde{\ZZ}X}$ on $\tilde{\ZZ}X$ (defined by $m=\id\eta\id$), for any $X\in\Space_*$.

\begin{defn}
\label{defn:cobar_construction}
Let $Y$ be an object in $\coAlgK$. The \emph{cosimplicial cobar construction} (or two-sided cosimplicial cobar construction) $C(Y):=\Cobar(U,\K,Y)$ looks like (showing only the coface maps)
\begin{align}
\label{eq:cobar_construction}
&\xymatrix@1{
  C(Y): \quad\quad
  UY\ar@<0.5ex>[r]^-{d^0}\ar@<-0.5ex>[r]_-{d^1} &
  U\K Y
  \ar@<1.0ex>[r]\ar[r]\ar@<-1.0ex>[r] &
  U\K\K Y
  \cdots
}
\end{align}
and is defined objectwise by $C(Y)^n:=U\K^n Y$ with the obvious coface and codegeneracy maps; see, for instance, the face and degeneracy maps in the simplicial bar constructions described in \cite[A.1]{Gugenheim_May} or \cite[Section 7]{May_classifying_spaces}, and dualize. For instance, in \eqref{eq:cobar_construction} the indicated coface maps are defined by $d^0:=m\id$ and $d^1:=\id m$.
\end{defn}

\subsection{Connectivity estimates, cofinality, and cubical diagrams}
\label{sec:cubical_diagrams_homotopical_analysis}

The purpose of this section is to prove Theorems \ref{thm:estimating_connectivity_of_maps_in_tower_C_of_Y} and \ref{thm:connectivities_for_map_that_commutes_Z_into_inside_of_holim} that provide the estimates we need. The following definitions appear in \cite[Section 1, 1.12]{Goodwillie_calculus_2} in the context of spaces.
\begin{defn}[Indexing categories for cubical diagrams]
Let $W$ be a finite set and $\M$ a category.
\begin{itemize}
\item Denote by $\powerset(W)$ the poset of all subsets of $W$, ordered by inclusion $\subset$ of sets. We will often regard $\powerset(W)$ as the category associated to this partial order in the usual way; the objects are the elements of $\powerset(W)$, and there is a morphism $U\rarrow V$ if and only if $U\subset V$.
\item Denote by $\powerset_0(W)\subset\powerset(W)$ the poset of all nonempty subsets of $W$; it is the full subcategory of $\powerset(W)$ containing all objects except the initial object $\emptyset$.
\item A \emph{$W$-cube} $\capX$ in $\M$ is a $\powerset(W)$-shaped diagram $\capX$ in $\M$; in other words, a functor $\function{\capX}{\powerset(W)}{\M}$.
\end{itemize}
\end{defn}

\begin{rem}
If $\capX$ is a $W$-cube in $\M$ where $|W|=n$, we will sometimes refer to $\capX$ simply as an \emph{$n$-cube} in $\M$. In particular, a $0$-cube is an object in $\M$ and a $1$-cube is a morphism in $\M$.
\end{rem}

\begin{defn}[Faces of cubical diagrams]
Let $W$ be a finite set and $\M$ a category. Let $\capX$ be a $W$-cube in $\M$ and consider any subsets $U\subset V\subset W$. Denote by $\partial_U^V\capX$ the $(V-U)$-cube defined objectwise by
\begin{align*}
  T\mapsto(\partial_U^V\capX)_T:=\capX_{T\cup U},\quad\quad T\subset V-U.
\end{align*}
In other words, $\partial_U^V\capX$ is the $(V-U)$-cube formed by all maps in $\capX$ between $\capX_U$ and $\capX_V$. We say that $\partial_U^V\capX$ is a \emph{face} of $\capX$ of \emph{dimension} $|V-U|$.
\end{defn}

The following definitions appear in \cite[Section 2]{Dundas_relative_K_theory}, \cite[A.8.0.1, A.8.3.1]{Dundas_Goodwillie_McCarthy}.

\begin{defn}[Subcubes of cubical diagrams]
Let $T,W$ be finite sets such that $|T|\leq|W|$ and $\M$ a category. Let $\capX$ be a $W$-cube in $\M$. A \emph{$T$-subcube of $\capX$} is a $T$-cube resulting from the precomposite of $\capX$ along an injection $\function{\xi}{\powerset(T)}{\powerset(W)}$ satisfying that if $U,V\subset T$, then $\xi(U\cap V)=\xi(U)\cap\xi(V)$ and $\xi(U\cup V)=\xi(U)\cup\xi(V)$. If $|T|=d$, we will often refer to a $T$-subcube of $\capX$ simply as a \emph{$d$-subcube} of $\capX$.
\end{defn}

\begin{rem}
In general, not all subcubes of $\capX$ are faces of $\capX$. For instance, consider any 2-cube $\capX$. There are exactly four 1-dimensional faces of $\capX$, and exactly five 1-subcubes of $\capX$.
\end{rem}

\begin{defn}
Let $\function{f}{\NN}{\NN}$ be a function and $W$ a finite set. A $W$-cube $\capX$ is \emph{$f$-cartesian} (resp. \emph{$f$-cocartesian}) if each $d$-subcube of $\capX$ is $f(d)$-cartesian (resp. $f(d)$-cocartesian); here, $\NN$ denotes the non-negative integers.
\end{defn}

The following is Dundas' higher Hurewicz theorem for spaces and is proved in \cite[2.6]{Dundas_relative_K_theory}; see also the subsequent elaboration in \cite[A.8.3]{Dundas_Goodwillie_McCarthy}. It provides an alternate proof, together with strong estimates for the uniform cartesian-ness of cubes built by iterations of the Hurewicz map, of the result in Bousfield-Kan \cite[III.5.4]{Bousfield_Kan} that the $\ZZ$-completion map $X\rarrow X^\wedge_\ZZ$ is a weak equivalence for any $1$-connected space $X$. These uniform cartesian-ness estimates, resulting from Proposition \ref{prop:higher_hurewicz_theorem}, will play a key role in our homotopical analysis of the derived counit map below.

\begin{prop}[Higher Hurewicz theorem]
\label{prop:higher_hurewicz_theorem}
Let $k\geq 2$. If $\capX$ is an $(\id+k)$-cartesian cube of pointed spaces, then so is $\capX\rarrow U\tilde{\ZZ}\capX$.
\end{prop}

\begin{defn}
\label{defn:coface_cube}
Let $n\geq -1$ and suppose $Z$ is a cosimplicial pointed space coaugmented by $\function{d^0}{Z^{-1}}{Z^0}$. The \emph{coface} $(n+1)$-cube, denoted $\capX_{n+1}$, associated to the coaugmented cosimplicial pointed space $Z^{-1}\rarrow Z$, is the canonical $(n+1)$-cube built from the coface relations \cite[I.1]{Goerss_Jardine}
$d^j d^i=d^i d^{j-1}$, if $i<j$, associated to the coface maps of the $n$-truncation
\begin{align*}
\xymatrix@1{
  Z^{-1}\ar[r]^-{d^0} &
  Z^0\ar@<-0.5ex>[r]_-{d^1}\ar@<0.5ex>[r]^-{d^0} &
  Z^1\ \cdots\ Z^n
}
\end{align*}
of $Z^{-1}\rarrow Z$; in particular, $\capX_0$ is the pointed space (or 0-cube) $Z^{-1}$.
\end{defn}

\begin{rem}
For instance, the coface 1-cube $\capX_1$ has the left-hand form
\begin{align*}
\xymatrix@1{
  Z^{-1}\ar[r]^-{d^0} & Z^0
}\quad\quad
\xymatrix@1{
  Z^{-1}\ar[r]^-{d^0}\ar[d]^-{d^0} & Z^0\ar[d]^-{d^0}\\
  Z^0\ar[r]^-{d^1} & Z^1
}
\end{align*}
and the coface 2-cube $\capX_2$ has the indicated right-hand form.
\end{rem}

The following proposition, proved in \cite[XI.9.2]{Bousfield_Kan}, allows one to compute homotopy limits of $\DD$-shaped diagrams in terms of homotopy limits over $\DD'$-shaped diagrams, provided that the comparison map $\DD'\rightarrow\DD$ is left cofinal \cite[XI.9.1]{Bousfield_Kan}.

\begin{prop}
\label{prop:cofinality_induces_weak_equivalence}
Let $\function{\alpha}{\DD'}{\DD}$ be a functor between small categories. If $Z$ is a $\DD$-shaped diagram in pointed spaces and $\alpha$ is left cofinal, then the induced map $\holim\nolimits_{\DD'} X\xleftarrow{\wequiv}\holim\nolimits_{\DD} X$ is a weak equivalence.
\end{prop}

\begin{defn}
Let $n\geq 0$. Denote by $\Delta^{\leq n}\subset\Delta$ the full subcategory of objects $[m]$ such that $m\leq n$ (see Proposition \ref{prop:holim_tower_with_truncated_delta_filtration}).
\end{defn}

The functor in the following definition, appearing in \cite[6.3]{Sinha_cosimplicial_models}, plays a key role in the homotopical analysis of this paper; see also \cite[9.4.1]{Munson_Volic_book_project}.

\begin{defn}
\label{defn:sinha_comparison_map}
Define the totally ordered sets $[n]:=\{0,1,\dotsc,n\}$ for each $n\geq 0$, and given their natural ordering. The functor $\capP_0([n])\rightarrow\Delta^{\leq n}$ is defined objectwise by $U\mapsto [|U|-1]$, and which sends $U\subset V$ in $\capP_0([n])$ to the composite 
\begin{align*}
  [|U|-1]\Iso U\subset V\Iso[|V|-1]
\end{align*}
where the indicated isomorphisms are the unique isomorphisms of totally ordered sets.
\end{defn}

\begin{rem}
For instance, the punctured 1-cube $\capP_0([0])\rightarrow\Delta^{\leq 0}$ has the left-hand form and the punctured 2-cube $\capP_0([1])\rightarrow\Delta^{\leq 1}$ has the indicated
\begin{align*}
\xymatrix@1{
  & \{0\}
}\quad\quad
\xymatrix@1{
  & \{1\}\ar[d]^-{d^0}\\
  \{0\}\ar[r]^-{d^1} & \{0,1\}
}
\end{align*}
right-hand form.
\end{rem}

The following proposition, proved in \cite[6.7]{Sinha_cosimplicial_models}, explains the homotopical significance of the punctured $n$-cube appearing in Definition \ref{defn:sinha_comparison_map}; see also  \cite[6.1--6.4]{Carlsson} and \cite[18.7]{Dugger_homotopy_colimits}.

\begin{prop}
\label{prop:left_cofinality_truncated_delta}
Let $n\geq 0$. The functor
$
  \capP_0([n])\rightarrow\Delta^{\leq n}
$
is left cofinal; hence, if $Z$ is a cosimplicial pointed space, then the induced map of the form
$
  \holim\nolimits_{\capP_0([n])}Z\xleftarrow{\wequiv}
  \holim\nolimits_{\Delta^{\leq n}}Z
$
is a weak equivalence (Proposition \ref{prop:cofinality_induces_weak_equivalence}).
\end{prop}

\begin{rem}[Higher Hurewicz implies the Bousfield-Kan result $X\wequiv X^\wedge_\ZZ$]
\label{rem:Dundas_proof}
Assume that $X$ is a 1-connected pointed space. For notational simplicity we often drop the forgetful functor $U$, appearing in Proposition \ref{prop:higher_hurewicz_theorem}, from our arguments. The $n$-truncation of the Bousfield-Kan cosimplicial resolution \eqref{eq:homology_resolution_introduction} has the form
\begin{align*}
\xymatrix@1{
  X\ar[r]^-{d^0} &
  \tilde{\ZZ}X\ar@<-0.5ex>[r]_-{d^1}\ar@<0.5ex>[r]^-{d^0} &
  \tilde{\ZZ}\tilde{\ZZ}X\ \cdots\ \tilde{\ZZ}^{n+1} X
}
\end{align*}
Dundas \cite[Section 2]{Dundas_relative_K_theory} points out that just as \eqref{eq:homology_resolution_introduction} is built by iterating the Hurewicz map, the associated coface $(n+1)$-cube $\capX_{n+1}$ can be built by applying the Hurewicz map to the coface $n$-cube $\capX_n$. In more detail: the coface $(n+1)$-cube $\capX_{n+1}$ can be described as the $(n+1)$-cube $\capX_n\rightarrow\tilde{\ZZ}\capX_n$ for each $n\geq 0$. 

To verify that the $\ZZ$-completion map $X\rightarrow X^\wedge_\ZZ$ is a weak equivalence, it suffices to verify that the map
\begin{align}
\label{eq:map_into_n_th_stage_Dundas'_proof_remark}
  X\rightarrow \holim\nolimits_{\Delta^{\leq n}}C(\tilde{\ZZ}X)
\end{align}
into the n-th stage of the homotopy limit tower has connectivity strictly increasing with $n$. The map \eqref{eq:map_into_n_th_stage_Dundas'_proof_remark} can be built, up to weak equivalence, from the coface $(n+1)$-cube $\capX_{n+1}$. In more detail: the map \eqref{eq:map_into_n_th_stage_Dundas'_proof_remark} can be described as the map $X\rightarrow\holim_{\powerset_0([n])}\capX_{n+1}$;  the connectivity of this map is the same as the cartesian-ness of the coface (n+1)-cube $\capX_{n+1}$, but this is the same as the cartesian-ness of the (n+1)-cube $\capX_n\rightarrow\tilde{\ZZ}\capX_n$, for each $n\geq 0$.

 Since $X$ is a 1-connected pointed space, the map $X\rightarrow *$ is 2-connected, and hence the 0-cube $\capX_0$ is $(\id+2)$-cartesian. Hence by Proposition \ref{prop:higher_hurewicz_theorem} we know that $\capX_1$ is $(\id+2)$-cartesian, and therefore another application of Proposition \ref{prop:higher_hurewicz_theorem} gives that $\capX_2$ is $(\id+2)$-cartesian, and so forth. In a similar way, the coface $(n+1)$ cube $\capX_{n+1}$ is $(\id+2)$-cartesian for each $n\geq 0$; hence Dundas' higher Hurewicz theorem has provided us with strong estimates for the uniform cartesian-ness of cubes built by iterations of the Hurewicz map. In particular, we know that the $(n+1)$-cube $\capX_{n+1}$ is $(n+1+2)$-cartesian for each $n\geq 0$, which means that the map \eqref{eq:map_into_n_th_stage_Dundas'_proof_remark} is $(n+3)$-connected for each $n\geq 0$. Therefore, these connectivity estimates imply that 
\begin{align*}
  X\rightarrow\holim\nolimits_n\holim\nolimits_{\Delta^{\leq n}}C(\tilde{\ZZ}X)\wequiv\holim\nolimits_\Delta C(\tilde{\ZZ}X)\wequiv X^\wedge_\ZZ
\end{align*}
is a weak equivalence; since this is the $\ZZ$-completion map $X\rightarrow X^\wedge_\ZZ$, we have recovered the Bousfield-Kan result. The uniform cartesian-ness estimates for $\capX_{n+1}$ are stronger than the statement that the coaugmentation $X\wequiv X^\wedge_\ZZ$ is a weak equivalence. For instance, such uniform cartesian-ness estimates imply uniform cocartesian-ness estimates, and vice-versa \cite[2.4]{Dundas_relative_K_theory}; the strength of these uniform cartesian-ness estimates become important in \cite{Dundas_relative_K_theory, Dundas_Goodwillie_McCarthy} and for the main results of this paper.
\end{rem}

The following is proved in \cite[3.4.8]{Munson_Volic_book_project}.

\begin{prop}
\label{prop:retraction_two_cube_argument}
Consider any $2$-cube $\capX$ of the form
\begin{align*}
\xymatrix@1{
  X\ar@{=}[d]\ar[r]^-{d} & Y\ar[d]^-{s}\\
  X\ar@{=}[r] & X
}
\end{align*}
in $\Space_*$; in other words, suppose $s$ is a retraction of $d$. There are natural weak equivalences $\hofib(d)\wequiv\Omega\hofib(s)$; here, the notation $d$ and $s$ is intended to suggest to the reader ``coface map'' and ``codegeneracy map'', respectively.
\end{prop}

\begin{defn}
Let $Z$ be a cosimplicial pointed space and $n\geq 0$. The \emph{codegeneracy} $n$-cube, denoted $\capY_n$, associated to $Z$, is the canonical $n$-cube built from the codegeneracy relations \cite[I.1]{Goerss_Jardine}
$s^j s^i=s^i s^{j+1}$, if $i\leq j$, associated to the codegeneracy maps of the $n$-truncation 
\begin{align*}
\xymatrix@1{
  Z^0 &
  Z^1
  \ar[l]_-{s^0} &
  Z^2\ar@<-0.5ex>[l]_-{s^0}\ar@<0.5ex>[l]^-{s^1}
  \ \cdots\ Z^n
}
\end{align*}
of $Z$; in particular, $\capY_0$ is the pointed space (or $0$-cube) $Z^0$. 
\end{defn}

\begin{rem}
For instance, the codegeneracy 1-cube $\capY_1$ has the left-hand form
\begin{align*}
\xymatrix@1{
  Z^1\ar[r]^-{s^0} & Z^0
}\quad\quad
\xymatrix@1{
  Z^{2}\ar[r]^-{s^1}\ar[d]^-{s^0} & Z^1\ar[d]^-{s^0}\\
  Z^1\ar[r]^-{s^0} & Z^0
}
\end{align*}
and the codegeneracy 2-cube $\capY_2$ has the indicated right-hand form.
\end{rem}

\begin{rem}
It is important to note that the total homotopy fiber of an $n$-cube of pointed spaces is weakly equivalent to its iterated homotopy fiber \cite[Section 1]{Goodwillie_calculus_2}, and in this paper we use the terms interchangeably; we use the convention that the iterated homotopy fiber of a $0$-cube $\capY$ (or object $\capY_\emptyset$) is the homotopy fiber of the unique map $\capY_\emptyset\rarrow *$ and hence is weakly equivalent to $\capY_\emptyset$; see also \cite[5.5.4]{Munson_Volic_book_project}.
\end{rem}

\begin{rem}
The homotopical significance of the codegeneracy $n$-cubes $\capY_n$ associated to a cosimplicial pointed space $Z$ can be understood as follows: the total homotopy fiber of $\capY_n$ is the derived version of the fiber of the natural map $Z^n\rightarrow M^{n-1}Z$; here, $M^{n-1}Z$ denotes the indicated matching space of $Z$ (\cite[X.4.5]{Bousfield_Kan} and \cite[VII.4.9]{Goerss_Jardine}); i.e., if $Z$ is Reedy fibrant, then there are natural weak equivalences \cite[X.6.3]{Bousfield_Kan}
$(\iter\hofib)\capY_n\wequiv \fiber(Z^n\rightarrow M^{n-1}Z)$, $n\geq 0$.
\end{rem}

The following calculation is proved in \cite[X.6.3]{Bousfield_Kan} for the $\Tot$ tower of a Reedy fibrant cosimplicial pointed space; compare with \cite[5.5.7]{Munson_Volic_book_project}.

\begin{prop}
\label{prop:iterated_homotopy_fibers_calculation}
Let $Z$ be a cosimplicial pointed space and $n\geq 0$. There are natural zigzags of weak equivalences
\begin{align*}
  \hofib(\holim\nolimits_{\Delta^{\leq n}}Z\rarrow\holim\nolimits_{\Delta^{\leq n-1}}Z)
  \wequiv\Omega^n(\iter\hofib)\capY_n
\end{align*}
where $\capY_n$ denotes the codegeneracy $n$-cube associated to $Z$.
\end{prop}

\begin{defn}
\label{defn:the_wide_tilde_construction}
Let $Z$ be an objectwise fibrant cosimplicial pointed space and $n\geq 0$. Denote by $\function{Z}{\capP_0([n])}{\Space_*}$ the corresponding composite diagram
$
  \capP_0([n])\rightarrow\Delta^{\leq n}
  \subset\Delta
  \xrightarrow{Z}\Space_*
$
(Definition \ref{defn:sinha_comparison_map}). The \emph{associated $\infty$-cartesian $(n+1)$-cube built from $Z$}, denoted $\function{\widetilde{Z}}{\capP([n])}{\Space_*}$, is defined objectwise by
\begin{align*}
  \widetilde{Z}_V :=
  \left\{
    \begin{array}{rl}
    \holim_{T\neq\emptyset}Z_T,&\text{for $V=\emptyset$,}\\
    Z_V,&\text{for $V\neq\emptyset$}.
    \end{array}
  \right.
\end{align*}
In other words, the $\widetilde{Z}$ construction is simply ``filling in'' the punctured $(n+1)$-cube $\function{Z}{\capP_0([n])}{\Space_*}$ with value $\widetilde{Z}_\emptyset=\holim\nolimits_{\capP_0([n])}Z\wequiv\holim\nolimits_{\Delta^{\leq n}}Z$ at the initial vertex to turn it into an $(n+1)$-cube that is $\infty$-cartesian. 
\end{defn}

\begin{rem}
For instance, in the case $n=1$ the $\widetilde{Z}$ construction produces the $\infty$-cartesian 2-cube of the form
\begin{align*}
\xymatrix@1{
  \holim\nolimits_{\Delta^{\leq 1}}Z\ar[r]\ar[d] & Z^0\ar[d]^-{d^0}\\
  Z^0\ar[r]^-{d^1} & Z^1
}
\end{align*}
\end{rem}

Let $Y$ be a 1-connected cofibrant $\K$-coalgebra. We want to estimate the connectivity of the map
\begin{align}
  \tilde{\ZZ}\holim\nolimits_{\Delta^{\leq n}}C(Y)
  \rightarrow
  \holim\nolimits_{\Delta^{\leq n}}\tilde{\ZZ}C(Y)
\end{align}
for each $n\geq 0$. In the case $n=0$ this is the identity map $\tilde{\ZZ}UY\rightarrow\tilde{\ZZ}UY$ and hence a weak equivalence. Consider the case $n=1$. Let's build $\widetilde{C(Y)}$, the $\infty$-cartesian 2-cube of the left-hand form
\begin{align*}
\xymatrix@1{
  \holim\nolimits_{\Delta^{\leq 1}}C(Y)\ar[r]\ar[d] & 
  C(Y)^0\ar[d]^-{d^0}\\
  C(Y)^0\ar[r]^-{d^1} & C(Y)^1
}\quad\quad
\xymatrix@1{
  \tilde{\ZZ}\holim\nolimits_{\Delta^{\leq 1}}C(Y)\ar[r]\ar[d] & 
  \tilde{\ZZ}C(Y)^0\ar[d]^-{\tilde{\ZZ} d^0}\\
  \tilde{\ZZ}C(Y)^0\ar[r]^-{\tilde{\ZZ} d^1} & 
  \tilde{\ZZ}C(Y)^1
}
\end{align*}
Applying $\tilde{\ZZ}$ gives the 2-cube $\tilde{\ZZ}\widetilde{C(Y)}$ of the indicated right-hand form. The connectivity of the map
\begin{align}
\label{eq:comparison_map_n_equals_1}
  \tilde{\ZZ}\holim\nolimits_{\Delta^{\leq 1}}C(Y)
  \rightarrow
  \holim\nolimits_{\Delta^{\leq 1}}\tilde{\ZZ}C(Y)
\end{align}
is the same as the cartesian-ness of the 2-cube $\tilde{\ZZ}\widetilde{C(Y)}$. The idea is to (i) estimate the cocartesian-ness of the 2-cube $\widetilde{C(Y)}$, (ii) applying $\tilde{\ZZ}$ will play nicely with the cocartesian-ness estimate, (iii) $\tilde{\ZZ}\widetilde{C(Y)}$ is a 2-cube in simplicial abelian groups, hence by \cite[3.10]{Ching_Harper} it is $k$-cocartesian if and only if it is $(k-2+1)$-cartesian. To carry this out, the idea is to use Goodwillie's higher dual Blakers-Massey theorem \cite[2.6]{Goodwillie_calculus_2}, which we recall here for the convenience of the reader, to estimate the cocartesian-ness of the 2-cube $\widetilde{C(Y)}$.

\begin{prop}[Higher dual Blakers-Massey theorem]
\label{prop:higher_dual_blakers_massey}
Let $W$ be a nonempty finite set. Let $\capX$ be a $W$-cube of pointed spaces. Suppose that
\begin{itemize}
\item[(i)] for each nonempty subset $V\subset W$, the $V$-cube $\partial_{W-V}^W\capX$ (formed by all maps in $\capX$ between $\capX_{W-V}$ and $\capX_W$) is $k_V$-cartesian,
\item[(ii)] $k_{U}\leq k_V$ for each $U\subset V$.
\end{itemize}
Then $\capX$ is $k$-cocartesian, where $k$ is the minimum of $|W|-1+\sum_{V\in\lambda}k_V$ over all partitions $\lambda$ of $W$ by nonempty sets.
\end{prop}

Taking $W=\{0,1\}$ since $\widetilde{C(Y)}$ is a 2-cube, the input to Proposition \ref{prop:higher_dual_blakers_massey} requires that we estimate the cartesian-ness of each of the faces
\begin{align*}
  \partial_{W-V}^W\widetilde{C(Y)},\quad\quad
  \emptyset\neq V\subset W
\end{align*}
Hence we need to estimate the cartesian-ness of the two 1-faces indicated in the left-hand diagram
\begin{align*}
\xymatrix@1{
   & 
  C(Y)^0\ar[d]^-{d^0}\\
  C(Y)^0\ar[r]^-{d^1} & C(Y)^1
}\quad\quad
\xymatrix@1{
   & 
  UY\ar[d]^-{d^0}\\
  UY\ar[r]^-{d^1} & U\tilde{\ZZ}UY
}
\end{align*}
which have the form in the indicated right-hand diagram. We know that $d^0=m\id$ is the Hurewicz map on $UY$, and since $UY$ is 1-connected we know that $d^0$ is a 3-connected map and hence a 3-cartesian 1-cube. What about the map $d^1=\id m$ involving the $K$-coaction map on $Y$? The key observation is that the cosimplicial identities force a certain ``uniformity of faces'' behavior as follows. Consider the commutative diagrams (or 2-cubes) of the form
\begin{align}
\label{eq:sequence_paths_1_cubes}
\xymatrix@1{
  UY\ar@{=}[d]\ar[r]^-{d^0} & 
  U\tilde{\ZZ}UY\ar[d]^-{s^0}\\
  UY\ar@{=}[r] & UY
}\quad\quad
\xymatrix@1{
  UY\ar@{=}[d]\ar[r]^-{d^1} & 
  U\tilde{\ZZ}UY\ar[d]^-{s^0}\\
  UY\ar@{=}[r] & UY
}
\end{align}
coming from the cosimplicial identities \cite[I.1]{Goerss_Jardine}. 
Then by Proposition \ref{prop:retraction_two_cube_argument} we know 
\begin{align*}
  \hofib(d^0)\wequiv\Omega\hofib(s^0),\quad\quad
  \hofib(d^1)\wequiv\Omega\hofib(s^0),
\end{align*}
and hence $\hofib(d^0)\wequiv\hofib(d^1)$. Therefore, by this uniformity we know that $d^1$ is also a 3-connected map and hence a 3-cartesian 1-cube. Since we know that the 2-face of $\widetilde{C(Y)}$ is $\infty$-cartesian (by construction), it follows from Proposition \ref{prop:higher_dual_blakers_massey} that $\widetilde{C(Y)}$ is $k$-cocartesian, where $k-1$ is the minimum of
\begin{align*}
  k_{\{0,1\}} = \infty,\quad\quad
  k_{\{0\}}+k_{\{1\}}&= 3+3 = 6.
\end{align*}
Hence $k=7$ and we have calculated that $\widetilde{C(Y)}$ is a 7-cocartesian 2-cube in $\Space_*$, hence $\tilde{\ZZ}\widetilde{C(Y)}$ is a 7-cocartesian 2-cube in $\sAb$, and therefore by above it is a $(7-1)$-cartesian 2-cube in $\sAb$. The upshot is that $\tilde{\ZZ}\widetilde{C(Y)}$ is 6-cartesian and hence we have calculated that the map \eqref{eq:comparison_map_n_equals_1} is 6-connected.

Consider the case $n=2$. Let's build the $\infty$-cartesian 3-cube $\widetilde{C(Y)}$. Applying $\tilde{\ZZ}$ gives the 3-cube $\tilde{\ZZ}\widetilde{C(Y)}$ and the connectivity of the map
\begin{align}
\label{eq:comparison_map_n_equals_2}
  \tilde{\ZZ}\holim\nolimits_{\Delta^{\leq 2}}C(Y)
  \rightarrow
  \holim\nolimits_{\Delta^{\leq 2}}\tilde{\ZZ}C(Y)
\end{align}
is the same as the cartesian-ness of $\tilde{\ZZ}\widetilde{C(Y)}$. The idea is to (i) estimate the cocartesian-ness of the 3-cube $\widetilde{C(Y)}$, (ii) applying $\tilde{\ZZ}$ will play nicely with the cocartesian-ness estimate, (iii) $\tilde{\ZZ}\widetilde{C(Y)}$ is a 3-cube in simplicial abelian groups, hence by \cite[3.10]{Ching_Harper} it is $k$-cocartesian if and only if it is $(k-3+1)$-cartesian. To carry this out, the idea is to use Proposition \ref{prop:higher_dual_blakers_massey} to estimate the cocartesian-ness of the 3-cube $\widetilde{C(Y)}$.

Taking $W=\{0,1,2\}$ since $\widetilde{C(Y)}$ is a 3-cube, the input to Proposition \ref{prop:higher_dual_blakers_massey} requires that we estimate the cartesian-ness of each of the faces
$\partial_{W-V}^W\widetilde{C(Y)}$,
$\emptyset\neq V\subset W$.
Hence we need to estimate the cartesian-ness of three 2-faces and three 1-faces (or maps). The key observation is that exactly one of these 2-faces does not involve the $\K$-coaction map on $Y$; furthermore, this particular 2-face is precisely the coface 2-cube $\capX_2$ in Remark \ref{rem:Dundas_proof} when taking $X=UY$. Since $UY$ is 1-connected, we know by Dundas' higher Hurewicz theorem and Remark \ref{rem:Dundas_proof} that $\capX_2$ is an $(\id+2)$-cartesian 2-cube; in particular, $\capX_2$ is 4-cartesian. What about the other two 2-faces involving the $K$-coaction map on $Y$? The key observation is that the cosimplicial identities force a certain ``uniformity of faces'' behavior as follows. For ease of notational purposes, let $Z=C(Y)$ and consider the commutative diagrams of the form
\begin{align}
\label{eq:sequence_paths_of_squares}
\xymatrix@1{
  Z^0\ar[d]_-{d^0}\ar[r]^-{d^0}\ar@{}[dr]|(0.43){(F_1)} &
  Z^1\ar[d]^-{d^1}\\
  Z^1\ar[d]_-{s^0}\ar[r]^-{d^0} &
  Z^2\ar[d]^-{s^1}\ar[r]^-{s^0} &
  Z^1\ar[d]^-{s^0}\\
  Z^0\ar[r]^-{d^0} &
  Z^1\ar[r]^-{s^0} &
  Z^0
}\quad\quad
\xymatrix@1{
  Z^0\ar[d]_-{d^1}\ar[r]^-{d^1}\ar@{}[dr]|(0.43){(F_2)} &
  Z^1\ar[d]^-{d^2}\ar[r]^-{s^0} &
  Z^0\ar[d]^-{d^1}\\
  Z^1\ar[r]^-{d^1} &
  Z^2\ar[d]^-{s^1}\ar[r]^-{s^0} &
  Z^1\ar[d]^-{s^0}\\
& Z^1\ar[r]^-{s^0} &
  Z^0
}\quad\quad
\xymatrix@1{
  Z^0\ar[d]_-{d^1}\ar[r]^-{d^0}\ar@{}[dr]|(0.43){(F_3)} &
  Z^1\ar[d]^-{d^2}\ar[r]^-{s^0} &
  Z^0\ar[d]^-{d^1}\\
  Z^1\ar[r]^-{d^0} &
  Z^2\ar[d]^-{s^1}\ar[r]^-{s^0} &
  Z^1\ar[d]^-{s^0}\\
   &
  Z^1\ar[r]^-{s^0} &
  Z^0
}
\end{align}
coming from the cosimplicial identities \cite[I.1]{Goerss_Jardine}. The upper left-hand square $(F_1)$ is the coface 2-cube $\capX_2$ which is $(\id+2)$-cartesian by above. The upper left-hand squares $(F_2)$ and $(F_3)$ are the remaining two 2-faces that we need cartesian-ness estimates for. The key observation is that the lower right-hand squares are each a copy of the codegeneracy $2$-cube $\capY_2$ associated to $Z$, and that furthermore, the indicated vertical and horizontal composites are the identity maps by the cosimplicial identities \cite[I.1]{Goerss_Jardine}; then by repeated application of Proposition \ref{prop:retraction_two_cube_argument} to these composites in \eqref{eq:sequence_paths_of_squares}, we know that
\begin{align*}
  (\iter\hofib)(F_1)&\wequiv
  \Omega^2(\iter\hofib)\capY_2\\
  (\iter\hofib)(F_2)&\wequiv
  \Omega^2(\iter\hofib)\capY_2\\
  (\iter\hofib)(F_3)&\wequiv
  \Omega^2(\iter\hofib)\capY_2
\end{align*}
and hence $(\iter\hofib)(F_1)\wequiv(\iter\hofib)(F_2)\wequiv(\iter\hofib)(F_3)$. Therefore, by this uniformity we know that $(F_2)$ and $(F_3)$ are also 4-cartesian 2-cubes. Similarly, we know that the three 1-faces (or maps) with codomain $Z^2$ are 3-cartesian. Since we know that the 3-face of $\widetilde{C(Y)}$ is $\infty$-cartesian (by construction), it follows from Proposition \ref{prop:higher_dual_blakers_massey} that $\widetilde{C(Y)}$ is $k$-cocartesian, where $k-2$ is the minimum of
\begin{align*}
  k_{\{0,1,2\}} = \infty,\quad
  k_{\{0\}}+k_{\{1,2\}} = 3+4=7,\quad
  k_{\{0\}}+k_{\{1\}}+k_{\{2\}}&= 3+3+3 = 9.
\end{align*}
Note that by the ``uniformity of faces'' behavior, we get nothing new from the other partitions of $W$; this is why we have not written them out here. Hence $k=9$ and we have calculated that $\widetilde{C(Y)}$ is a 9-cocartesian 3-cube in $\Space_*$, hence $\tilde{\ZZ}\widetilde{C(Y)}$ is a 9-cocartesian 3-cube in $\sAb$, and therefore by above it is a $(9-2)$-cartesian 3-cube in $\sAb$. The upshot is that $\tilde{\ZZ}\widetilde{C(Y)}$ is 7-cartesian and hence we have calculated that the map \eqref{eq:comparison_map_n_equals_2} is 7-connected. 

\begin{rem}
\label{rem:more_information_in_arguemnt_above}
There is more information in the argument above. Since the 2-face $(F_1)$ is 4-cartesian, its total homotopy fiber is 3-connected, hence (Proposition \ref{prop:iterated_homotopy_fibers_calculation})
\begin{align*}
  \hofib(\holim\nolimits_{\Delta^{\leq 2}}Z\rarrow\holim\nolimits_{\Delta^{\leq 1}}Z)
  \wequiv\Omega^2(\iter\hofib)\capY_2
\end{align*}
is 3-connected and therefore the map $\holim_{\Delta^{\leq 2}}Z\rarrow\holim_{\Delta^{\leq 1}}Z$ is 4-connected. Also, since $\Omega^2(\iter\hofib)\capY_2$ is 3-connected, then $(\iter\hofib)\capY_2$ is 5-connected. 
\end{rem}

And so forth, in a similar way, for each $n\geq 3$, the connectivity of the map
\begin{align*}
  \tilde{\ZZ}\holim\nolimits_{\Delta^{\leq n}}C(Y)
  \rightarrow
  \holim\nolimits_{\Delta^{\leq n}}\tilde{\ZZ}C(Y)
\end{align*}
is the same as the cartesian-ness of the (n+1)-cube $\tilde{\ZZ}\widetilde{C(Y)}$. The idea is to (i) estimate the cocartesian-ness of the (n+1)-cube $\widetilde{C(Y)}$, (ii) applying $\tilde{\ZZ}$ will play nicely with the cocartesian-ness estimate, (iii) $\tilde{\ZZ}\widetilde{C(Y)}$ is an (n+1)-cube in simplicial abelian groups, hence by \cite[3.10]{Ching_Harper} it is $k$-cocartesian if and only if it is $(k-(n+1)+1)$-cartesian. To carry this out, the idea is to use Proposition \ref{prop:higher_dual_blakers_massey} to estimate the cocartesian-ness of the (n+1)-cube $\widetilde{C(Y)}$. We can organize our argument as follows, exactly as in the above cases for $n=1,2$.

First we recall the following proposition, which appears in \cite[7.31]{Ching_Harper_derived_Koszul_duality}; it can be proved by arguing exactly as in \cite[5.5.7]{Munson_Volic_book_project}. We have already verified it above in low dimensional cases; see \eqref{eq:sequence_paths_1_cubes} and \eqref{eq:sequence_paths_of_squares}, together with the resulting iterated  homotopy fiber calculations. 

\begin{prop}[Uniformity of faces]
\label{prop:comparing_faces_of_coface_cube_with_codegeneracy_cube}
Let $Z$ be a cosimplicial pointed space and $n\geq 0$. Let $\emptyset\neq T\subset[n]$ and $t\in T$. Then there is a weak equivalence
\begin{align*}
  (\iter\hofib)\partial_{\{t\}}^T\widetilde{Z}\wequiv
  \Omega^{|T|-1}(\iter\hofib)\capY_{|T|-1}
\end{align*}
in $\Space_*$, where $\capY_{|T|-1}$ denotes the codegeneracy $(|T|-1)$-cube associated to $Z$.
\end{prop}

\begin{rem}
We will exploit Proposition \ref{prop:comparing_faces_of_coface_cube_with_codegeneracy_cube} below by taking $Z=C(Y)$ for $Y$ a cofibrant 1-connected $\K$-coalgebra. It follows from the observation that each face $\partial_{\{t\}}^T\widetilde{Z}$ of the $\widetilde{Z}$ construction is connected to the codegeneracy cube $\capY_{|T|-1}$ by a ``sequence of retractions'' built from codegeneracy maps: see \eqref{eq:sequence_paths_1_cubes} for the case of 1-faces and \eqref{eq:sequence_paths_of_squares} for the case of 2-faces; the higher dimensional faces are similar, and the argument is then completed by repeated application of Proposition \ref{prop:retraction_two_cube_argument}; see \cite[7.34]{Ching_Harper_derived_Koszul_duality}.
\end{rem}

\begin{thm}
\label{thm:cocartesian_and_cartesian_estimates}
Let $Y$ be a 1-connected cofibrant $\K$-coalgebra and $n\geq 1$. Consider the $\infty$-cartesian $(n+1)$-cube $\widetilde{C(Y)}$ in $\Space_*$ built from $C(Y)$. Then
\begin{itemize}
\item[(a)] the cube $\widetilde{C(Y)}$ is $(2n+5)$-cocartesian in $\Space_*$,
\item[(b)] the cube $\tilde{\ZZ}\widetilde{C(Y)}$ is $(2n+5)$-cocartesian in $\sAb$,
\item[(c)] the cube $\tilde{\ZZ}\widetilde{C(Y)}$ is $(n+5)$-cartesian in $\sAb$.
\end{itemize}
\end{thm}

\begin{proof}
Consider part (a). Taking $W=\{0,1,\dots, n\}$ since $\widetilde{C(Y)}$ is an $(n+1)$-cube, our strategy is to use Goodwillie's higher dual Blakers-Massey theorem (Proposition \ref{prop:higher_dual_blakers_massey}) to estimate how close the $W$-cube $\widetilde{C(Y)}$ in $\Space_*$ is to being cocartesian; the input to Proposition \ref{prop:higher_dual_blakers_massey} requires that we estimate the cartesian-ness of each of the faces
\begin{align*}
  \partial_{W-V}^W\widetilde{C(Y)},\quad\quad
  \emptyset\neq V\subset W
\end{align*}
 We know from Dundas' higher Hurewicz theorem (Proposition \ref{prop:higher_hurewicz_theorem}), on iterations of the Hurewicz map applied to $UY$, together with the ``uniformity of faces'' property enforced by the cosimplicial identities and summarized in Proposition \ref{prop:comparing_faces_of_coface_cube_with_codegeneracy_cube}, that for each nonempty subset $V\subset W$, the $V$-cube $\partial_{W-V}^W\widetilde{C(Y)}$ is $(|V|+2)$-cartesian; since it is $\infty$-cartesian by construction when $V=W$, it follows immediately from Proposition \ref{prop:higher_dual_blakers_massey} that $\widetilde{C(Y)}$ is $(2n+5)$-cocartesian in $\Space_*$, which finishes the proof of part (a). Part (b) follows from the fact that $\function{\tilde{\ZZ}}{\Space_*}{\sAb}$ is a left Quillen functor together with the fact that $\tilde{\ZZ}$ preserves connectivity of maps between $1$-connected spaces. Part (c) follows easily from the fact that $\sAb$ and $\Chaincx_{\geq 0}(\ZZ)$ are Quillen equivalent via the normalization functor and that the homotopy groups of a simplicial abelian group are naturally isomorphic to their associated homology groups as chain complexes, together with the obvious chain complex analog of \cite[3.10]{Ching_Harper}; in other words, that $\tilde{\ZZ}\widetilde{C(Y)}$ is $k$-cocartesian if and only if it is $(k-n)$-cartesian. Taking $k=(2n+5)$ from part (b) finishes the proof.
\end{proof}

\begin{proof}[Proof of Theorem \ref{thm:connectivities_for_map_that_commutes_Z_into_inside_of_holim}]
We want to estimate how connected the comparison map
$
  \tilde{\ZZ}\holim\nolimits_{\Delta^{\leq n}} C(Y)\rightarrow
  \holim\nolimits_{\Delta^{\leq n}} \tilde{\ZZ}\,C(Y)
$
is, which is equivalent to estimating how cartesian $\tilde{\ZZ}\widetilde{C(Y)}$ is, and Theorem \ref{thm:cocartesian_and_cartesian_estimates}(c) completes the proof.
\end{proof}

Similar to Remark \ref{rem:more_information_in_arguemnt_above}, there is more information in the proof of Theorem \ref{thm:cocartesian_and_cartesian_estimates} above. We know that for exactly one $w\in W$, the $n$-face $\partial_{\{w\}}^W\widetilde{C(Y)}$ (i.e., the unique $n$-face of this form not involving the $\K$-coaction map on $Y$) in the proof of Theorem \ref{thm:cocartesian_and_cartesian_estimates} is precisely the coface $n$-cube $\capX_n$ in Remark \ref{rem:Dundas_proof} when taking $X=UY$. Since $UY$ is 1-connected, we know that $\capX_n$ is an $(\id+2)$-cartesian $n$-cube by the higher Hurewicz theorem; in particular, this $n$-face $\partial_{\{w\}}^W\widetilde{C(Y)}$ is $(n+2)$-cartesian and hence its total homotopy fiber is $(n+1)$-connected. By Proposition \ref{prop:comparing_faces_of_coface_cube_with_codegeneracy_cube}, we know
\begin{align*}
  (\iter\hofib)\partial_{\{w\}}^W\widetilde{C(Y)}\wequiv
  \Omega^{(n+1)-1}(\iter\hofib)\capY_{(n+1)-1}
\end{align*}
similar to Remark \ref{rem:more_information_in_arguemnt_above}, hence  by Proposition \ref{prop:iterated_homotopy_fibers_calculation} we know that
\begin{align*}
  \hofib(\holim\nolimits_{\Delta^{\leq n}}C(Y)\rarrow\holim\nolimits_{\Delta^{\leq n-1}}C(Y))
  \wequiv\Omega^n(\iter\hofib)\capY_n
\end{align*}
is $(n+1)$-connected; therefore the map 
$
  \holim\nolimits_{\Delta^{\leq n}}C(Y)\rarrow\holim\nolimits_{\Delta^{\leq n-1}}C(Y)
$
is $(n+2)$-connected. Also, since $\Omega^n(\iter\hofib)\capY_n$ is (n+1)-connected, then we know $(\iter\hofib)\capY_n$ is (2n+1)-connected. The upshot is that we have just proved Proposition \ref{prop:iterated_hofiber_codegeneracy_cube} and Theorem \ref{thm:estimating_connectivity_of_maps_in_tower_C_of_Y}.

\begin{prop}
\label{prop:iterated_hofiber_codegeneracy_cube}
Let $Y$ be a cofibrant $\K$-coalgebra and $n\geq 1$. Denote by $\capY_n$ the codegeneracy $n$-cube associated to the cosimplicial cobar construction $C(Y)$ of $Y$. If $Y$ is $1$-connected, then the total homotopy fiber of $\capY_n$ is $(2n+1)$-connected.
\end{prop}

\begin{proof}[Proof of Theorem \ref{thm:estimating_connectivity_of_maps_in_tower_C_of_Y}]
The homotopy fiber of 
$
  \holim\nolimits_{\Delta^{\leq n}}C(Y)\rarrow\holim\nolimits_{\Delta^{\leq n-1}}C(Y)
$
is weakly equivalent to $\Omega^n$ of the total homotopy fiber of the codegeneracy $n$-cube $\capY_{n}$ associated to $C(Y)$ by Proposition \ref{prop:iterated_homotopy_fibers_calculation}, hence by Proposition \ref{prop:iterated_hofiber_codegeneracy_cube} this map is $(n+2)$-connected.
\end{proof}

\subsection{Strong convergence for the $\holim_\Delta C(Y)$ spectral sequence}

The following strong convergence result for the homotopy spectral sequence (\cite[IX.4]{Bousfield_Kan}, \cite[VIII.1]{Goerss_Jardine}) associated to the cosimplicial cobar construction $C(Y)$ of a $\K$-coalgebra $Y$ is a corollary of the connectivity estimates in Theorem \ref{thm:estimating_connectivity_of_maps_in_tower_C_of_Y}.

\begin{thm}
\label{thm:strong_convergence_ss}
If $Y$ is a $1$-connected cofibrant $\K$-coalgebra, then the homotopy spectral sequence
$
  E^2_{-s,t} = \pi^s\pi_t C(Y)
  \Rightarrow
  \pi_{t-s}\holim\nolimits_\Delta C(Y)
$
converges strongly (Remark \ref{rem:strong_convergence}).
\end{thm}

\begin{proof}
This follows from the connectivity estimates in Theorem \ref{thm:estimating_connectivity_of_maps_in_tower_C_of_Y}.
\end{proof}

\begin{rem}
\label{rem:strong_convergence}
By \emph{strong convergence} of $\{E^r\}$ to $\pi_*\holim\nolimits_\Delta C(Y)$ we mean that (i) for each $(-s,t)$, there exists an $r$ such that $E^r_{-s,t}=E^\infty_{-s,t}$ and (ii) for each $i$, $E^\infty_{-s,s+i}=0$ except for finitely many $s$. Strong convergence implies that for each $i$, $\{E^\infty_{-s,s+i}\}$ is the set of filtration quotients from a finite filtration of $\pi_i\holim\nolimits_\Delta C(Y)$; see, for instance, \cite[IV.5.6, IX.5.3, IX.5.4]{Bousfield_Kan} and \cite[p. 255]{Dwyer_strong_convergence}.
\end{rem}

This is the homotopy spectral sequence associated to the cosimplicial cobar construction \eqref{defn:cobar_construction}; it generalizes to $\K$-coalgebra complexes the unstable Adams spectral sequence of a space; see \cite{Bousfield_Kan_spectral_sequence} and the subsequent work of \cite{Bendersky_Curtis_Miller, Bendersky_Thompson}.

\section{Background on simplicial structures}
\label{sec:simplicial_structures}

In this section we recall the simplicial structure on pointed spaces and simplicial abelian groups; the expert may wish to skim through, or skip entirely, this background section.

\begin{defn}
\label{defn:simplicial_structure_pointed_spaces}
Let $X,X'$ be pointed spaces and $K$ a simplicial set. The \emph{tensor product} $X\tensordot K$ in $\Space_*$, \emph{mapping space} $\Hombold_{\Space_*}(X,X')$ in $\sSet$, and \emph{mapping object} $\hombold_{\Space_*}(K,X)$ in $\Space_*$ are 
$X\tensordot K :=X\Smash K_+$, $\Hombold_{\Space_*}(X,X')_n := \hom_{\Space_*}(X\tensordot\Delta[n],X')$, and $\hombold_{\Space_*}(K,X')_n:=\hom_{\Space_*}(K_+\tensordot\Delta[n],X')$, 
where $\hombold_{\Space_*}(K,X')$ is pointed by the constant map.
\end{defn}

\begin{defn}
\label{defn:simplicial_structure_sAb}
Let $Y,Y'$ be simplicial abelian groups and $K$ a simplicial set. The \emph{tensor product} $Y\tensordot K$ in $\sAb$, \emph{mapping space} $\Hombold_{\sAb}(Y,Y')$ in $\sSet$, and \emph{mapping object} $\hombold_{\sAb}(K,Y)$ in $\sAb$ are defined by $Y\tensordot K:=Y\tensor\ZZ K$, $\Hombold_{\sAb}(Y,Y')_n := \hom_{sAb}(Y\tensordot\Delta[n],Y')$, and 
$\hombold_{\sAb}(K,Y')_n:=\hom_{sAb}(\ZZ K\tensordot\Delta[n],Y')$, where the mapping object $\hombold_{\sAb}(K,Y')$ inherits the usual abelian group structure from $Y'$. 
\end{defn}

For ease of notation purposes, we sometimes drop the $\Space_*$ and $\sAb$ decorations from the notation and simply write $\Hombold$ and $\hombold$.

\begin{prop}
With the above definitions of mapping object, tensor product, and mapping space the categories of pointed spaces $\Space_*$ and simplicial abelian groups $\sAb$ are  simplicial model categories.
\end{prop}

\begin{proof}
This is proved, for instance, in \cite[II.3]{Goerss_Jardine}.
\end{proof}

\begin{rem}
\label{rem:useful_adjunction_isomorphisms_simplicial_structure}
Let $\M$ denote either $\Space_*$ or $\sAb$. In particular, there are isomorphisms
\begin{align}
\label{eq:tensordot_adjunction_isomorphisms}
  \hom_{\M}(X\tensordot K,X')
  \Iso\hom_{\M}(X,\hombold(K,X'))
  \Iso\hom_\sSet(K,\Hombold(X,X'))
\end{align}
in $\Set$, natural in $X,K,X'$, that extend to isomorphisms
\begin{align*}
  \Hombold_{\M}(X\tensordot K,X')
  \Iso\Hombold_{\M}(X,\hombold(K,X'))
  \Iso\Hombold_\sSet(K,\Hombold(X,X'))
\end{align*}
in $\sSet$, natural in $X,K,X'$.
\end{rem}

Recall that the free-forgetful adjunction \eqref{eq:homology_adjunction}, whose unit is the space level Hurewicz map, is a Quillen adjunction with left adjoint on top and $U$ the forgetful functor; in particular, there is an isomorphism
$\hom_{\sAb}(\tilde{\ZZ}X,Y)\Iso\hom_{\Space_*}(X,UY)$
in $\Set$, natural in $X,Y$. The following proposition, which is proved in \cite[II.2.9]{Goerss_Jardine}, is fundamental to this paper. It verifies that the free-forgetful adjunction \eqref{eq:homology_adjunction} meshes nicely with the simplicial structure.

\begin{prop}
\label{prop:useful_properties_of_the_adjunction}
Let $X$ be a pointed space, $Y$ a simplicial abelian group, and $K,L$ simplicial sets.  Then
\begin{itemize}
\item[(a)] there is a natural isomorphism
$
  \sigma\colon\thinspace \tilde{\ZZ}(X)\tensordot K \xrightarrow{\Iso}\tilde{\ZZ}(X\tensordot K)
$;
\item[(b)] there is an isomorphism
$
  \Hombold(\tilde{\ZZ}X,Y)\Iso\Hombold(X,UY)
$
in $\sSet$, natural in $X,Y$, that extends the adjunction isomorphism associated to \eqref{eq:homology_adjunction};
\item[(c)] there is an isomorphism
$
  U\hombold(K,Y)\Iso\hombold(K,UY)
$
in $\Space_*$, natural in $K,Y$.
\item[(d)] there is a natural map
$
  \function{\sigma}{U(Y)\tensordot K}{U(Y\tensordot K)}
$
induced by $U$.
\item[(e)] the functors $\tilde{\ZZ}$ and $U$ are simplicial functors (Remark \ref{rem:simplicial_functors}) with the structure maps $\sigma$ of (a) and (d), respectively.
\end{itemize}
\end{prop}

\begin{rem}
\label{rem:simplicial_functors}
For a useful reference on simplicial functors in the context of homotopy theory, see \cite[9.8.5]{Hirschhorn}.
\end{rem}

The following proposition is fundamental to this paper.

\begin{prop}
\label{prop:unit_and_counit_are_simplicial}
Consider the monad $U\tilde{\ZZ}$ on $\Space_*$ and the comonad $\K=\tilde{\ZZ}U$ on $\sAb$ associated to the adjunction $(\tilde{\ZZ},U)$ in \eqref{eq:homology_adjunction}. The four associated natural transformations \eqref{eq:TQ_homology_spectrum_functor_natural_transformations}
are simplicial natural transformations.
\end{prop}

\begin{proof}
This is an exercise left to the reader; compare \cite[Proof of 3.16]{Ching_Harper_derived_Koszul_duality}.
\end{proof}

\section{Background on homotopy limits of $\Delta$-shaped diagrams}
\label{sec:homotopy_limit_towers}

The purpose of this section is to recall some well-known constructions and properties associated to the homotopy limit of $\Delta$-shaped diagrams; the expert may wish to skim through, or skip entirely, this background section.

\begin{rem}
From now on in this section, we assume that $\M$ is the simplicial model category $\Space_*$, $\sAb$, or $\sSet$ (see \cite[II.2]{Goerss_Jardine}) with tensor product $X\tensordot K$ in $\M$, mapping space $\Hombold(X,X')$ in $\sSet$, and mapping object $\hombold(K,X')$ in $\M$; here,  $X,X'\in\M$ and $K\in\sSet$.
\end{rem}

\begin{defn}
A cosimplicial object $Z\in\M^\Delta$ in $\M$ is \emph{coaugmented} if it comes with a map $\function{d^0}{Z^{-1}}{Z^0}$ in $\M$ such that $\function{d^0d^0=d^1d^0}{Z^{-1}}{Z^1}$; in this case, it follows easily from the cosimplicial identities (\cite[I.1]{Goerss_Jardine}) that $d^0$ induces a map
$Z^{-1}\rarrow Z$ of $\Delta$-shaped diagrams in $\M$, where $Z^{-1}$ denotes the constant cosimplicial object with value $Z^{-1}$; i.e., via the inclusion $Z^{-1}\in\M\subset\M^\Delta$ of constant diagrams.
\end{defn}

We follow Dror-Dwyer \cite[3.3]{Dror_Dwyer_long_homology}
in use of the terms \emph{restricted cosimplicial objects} for $\Delta_\res$-shaped diagrams, and \emph{restricted simplicial category} $\Delta_\res$ to denote the subcategory of $\Delta$ with objects the totally ordered sets $[n]$ for $n\geq 0$ and morphisms the strictly monotone maps of sets $\function{\xi}{[n]}{[n']}$; i.e., such that $k<l$ implies $\xi(k)<\xi(l)$.

\begin{defn}
\label{defn:totalization_and_restricted_totalization}
The \emph{totalization} functor $\Tot$ for cosimplicial objects in $\M$ and the \emph{restricted totalization} (or fat totalization) functor $\Tot^\res$ for restricted cosimplicial objects in $\M$ are defined objectwise by the ends
\begin{align*}
  \function{\Tot}{\M^\Delta}{\M},\quad\quad
  &X\mapsto\hombold(\Delta[-],X)^\Delta\\
  \function{\Tot^\res}{\M^{\Delta_\res}}{\M},\quad\quad
  &Y\mapsto\hombold(\Delta[-],Y)^{\Delta_\res}
\end{align*}
We often drop the adjective ``restricted'' and simply refer to both functors as \emph{totalization} functors. It follows from the universal property of ends that $\Tot(X)$ is naturally isomorphic to an equalizer diagram of the form
\begin{align*}
  \Tot(X)&\Iso
  \lim\Bigl(
  \xymatrix{
    \prod\limits_{[n]\in\Delta}\hombold(\Delta[n],X^n)
    \ar@<0.5ex>[r]\ar@<-0.5ex>[r] &
    \prod\limits_{\substack{[n]\rightarrow [n']\\ \text{in}\,\Delta}}\hombold(\Delta[n],X^{n'})
  }
  \Bigr)
\end{align*}
in $\M$, and similarly for $\Tot^\res(Y)$ by replacing $\Delta$ with $\Delta_\res$.  We sometimes refer to the natural maps $\Tot(X)\rarrow\hombold(\Delta[n],X^n)$ and $\Tot^\res(Y)\rarrow\hombold(\Delta[n],Y^n)$ as the \emph{projection} maps.
\end{defn}

\begin{prop}
\label{prop:tot_adjunctions}
The totalization functors $\Tot$ and $\Tot^\res$ fit into adjunctions
\begin{align}
\label{eqref:tot_adjunctions}
\xymatrix{
  \M\ar@<0.5ex>[r]^-{-\tensordot\Delta[-]} &
  \M^\Delta\ar@<0.5ex>[l]^-{\Tot}
},\quad\quad
\xymatrix{
  \M\ar@<0.5ex>[r]^-{-\tensordot\Delta[-]} &
  \M^{\Delta_\res}\ar@<0.5ex>[l]^-{\Tot^\res}
}
\end{align}
with left adjoints on top.
\end{prop}

\begin{defn}
Let $\DD$ be a small category. The \emph{Bousfield-Kan homotopy limit} $\holim_\DD^\BK$ for $\DD$-shaped diagrams in $\M$ is defined by, $
  \function{\holim\nolimits_\DD^\BK}{\M^\DD}{\M}$,
  $X\mapsto\Tot\xymatrix{\prod\nolimits^*_\DD} X$.
We will sometimes suppress $\DD$ from the notation and simply write $\holim^\BK$ and $\prod^*$. Here, the cosimplicial replacement functor $\function{\prod\nolimits^*}{\M^{\DD}}{\M^\Delta}$ is defined objectwise by (with the obvious coface $d^i$ and codegeneracy maps $s^j$)
\begin{align*}
\xymatrix{
  \prod^n X:= \prod\limits_{\substack{a_0\rightarrow\cdots\rightarrow a_n\\ \text{in}\,\DD}}X(a_n)
}
\end{align*}
\end{defn}

The simplicial category $\Delta$ has a natural filtration by its truncated subcategories $\Delta^{\leq n}$ of the form
$
  \emptyset\subset\Delta^{\leq 0}\subset\Delta^{\leq 1}
  \subset\cdots\subset\Delta^{\leq{n}}\subset
  \Delta^{\leq n+1}\subset\cdots\subset\Delta
$
where $\Delta^{\leq n}\subset\Delta$ denotes the full subcategory of objects $[m]$ such that $m\leq n$; we use the convention that $\Delta^{\leq -1}=\emptyset$ is the empty category. This leads to the following $\holim^\BK$ tower of a $\Delta$-shaped diagram in $\M$.

\begin{prop}
\label{prop:holim_tower_with_truncated_delta_filtration}
If $X\in\M^\Delta$, then $\holim\nolimits^\BK_\Delta X$ is naturally isomorphic to a limit of the form
$
  \holim\nolimits^\BK_\Delta X \Iso\lim
  \bigl({*}\larrow
  \holim\nolimits^\BK_{\Delta^{\leq 0}}X\larrow
  \holim\nolimits^\BK_{\Delta^{\leq 1}}X\larrow
  \holim\nolimits^\BK_{\Delta^{\leq 2}}X\larrow
  \cdots\bigr)
$
in $\M$; here, it may be helpful to note that 
$\holim\nolimits^\BK_{\Delta^{\leq -1}} X=*$ and $\holim\nolimits^\BK_{\Delta^{\leq 0}} X\Iso X^0$.
\end{prop}

\begin{defn}
\label{defn:totalization_functors}
Let $s\geq -1$. The functors $\Tot_s$ and $\Tot^\res_s$ are defined objectwise by the ends
\begin{align*}
  \function{\Tot_s}{\M^\Delta}{\M},\quad\quad
  &X\mapsto\hombold(\Sk_s\Delta[-],X)^\Delta\\
  \function{\Tot^\res_s}{\M^{\Delta_\res}}{\M},\quad\quad
  &Y\mapsto\hombold(\Sk_s\Delta[-],Y)^{\Delta_\res}
\end{align*}
Here we use the convention that the $(-1)$-skeleton of a simplicial set is the empty simplicial set. In particular, $\Sk_{-1}\Delta[n]=\emptyset$ for each $n\geq 0$; it follows immediately that $\Tot_{-1}(X)=*$ and $\Tot^\res_{-1}(Y)=*$.
\end{defn}

\begin{prop}
\label{prop:comparing_holim_with_Tot}
If $Y\in\M^\Delta$ is Reedy fibrant, then the natural maps
$
  \Tot Y\xrightarrow{\wequiv}\holim\nolimits_\Delta^\BK Y$ and 
$\Tot_n Y\xrightarrow{\wequiv}
  \holim\nolimits_{\Delta^{\leq n}}^\BK Y$
in $\M$ are weak equivalences.
\end{prop}

\begin{proof}
The left-hand map is the composite
\begin{align*}
  \hombold(\Delta[-],Y)^{\Delta}
  \xrightarrow{\wequiv}&\hombold(B(\Delta/-),Y)^\Delta
  \Iso\holim\nolimits^\BK_{\Delta}Y
\end{align*}
where the indicated map is a weak equivalence \cite[XI.4.4]{Bousfield_Kan} since it is induced by the natural map 
$\Delta[-]\xleftarrow{\wequiv}B(\Delta/-)$ in $(\sSet)^{\Delta}$, which itself is a weak equivalence between Reedy cofibrant objects \cite[XI.2.6]{Bousfield_Kan}; here, $B$ denotes the nerve functor. Similarly, the right-hand map is the composite
\begin{align*}
  \hombold(\Sk_n\Delta[-],Y)^\Delta
  \Iso&\hombold(\Delta^{\leq n}[-],Y)^{\Delta^{\leq n}}\\
  \xrightarrow
  {\wequiv}&\hombold(B(\Delta^{\leq n}/-),Y)^{\Delta^{\leq n}}\Iso\holim\nolimits^\BK_{\Delta^{\leq n}}Y
\end{align*}
where the indicated map is a weak equivalence \cite[XI.4.4]{Bousfield_Kan} since it is induced by the natural map 
$\Delta^{\leq n}[-]\xleftarrow{\wequiv}B(\Delta^{\leq n}/-)$ in $(\sSet)^{\Delta^{\leq n}}$, which itself is a weak equivalence between Reedy cofibrant objects \cite[XI.2.6]{Bousfield_Kan}.
\end{proof}

\begin{prop}
\label{prop:left_cofinal_delta_restricted_to_delta}
The inclusion of categories $\Delta_\res\subset\Delta$ is left cofinal; hence, if $X\in\M^\Delta$ is objectwise fibrant, then the induced map
$
  \holim^\BK\nolimits_{\Delta_\res}X\xleftarrow{\wequiv}
  \holim^\BK\nolimits_{\Delta}X
$
is a weak equivalence.
\end{prop}

\begin{proof}
The inclusion $\Delta_\res\subset\Delta$ is left cofinal by \cite[3.17]{Dror_Dwyer_long_homology}, hence the induced map on $\holim^\BK$ is a weak equivalence by \cite[XI.9.2]{Bousfield_Kan} (Proposition \ref{prop:cofinality_induces_weak_equivalence}). 
\end{proof}

\begin{prop}
\label{prop:tot_restricted_compared_with_holim_delta_restricted}
If $X\in\M^{\Delta_\res}$ is objectwise fibrant, then the natural map (in $\M$)
$
  \Tot^\res X\xrightarrow{\wequiv}
  \holim\nolimits_{\Delta_\res}^\BK X
$
is a weak equivalence.
\end{prop}

\begin{proof}
This is the same as in Proposition \ref{prop:comparing_holim_with_Tot}, except here the  map is the composite
\begin{align*}
  \hombold(\Delta[-],X)^{\Delta_\res}
  \xrightarrow{\wequiv}&\hombold(B(\Delta_\res/-),X)^{\Delta_\res}
  \Iso\holim\nolimits^\BK_{\Delta_\res}X
\end{align*}
where the indicated map is a weak equivalence \cite[XI.4.4]{Bousfield_Kan} since it is induced by the natural map 
$\Delta[-]\xleftarrow{\wequiv}B(\Delta_\res/-)$ in $(\sSet)^{\Delta_\res}$ .
\end{proof}

\begin{defn}
\label{defn:homotopy_limit_derived}
Let $\DD$ be a small category. The \emph{homotopy limit} functor $\holim_\DD$ for $\DD$-shaped diagrams is defined objectwise by
$
  \function{\holim\nolimits_\DD}{\M^\DD}{\M}$, 
$X\mapsto \holim\nolimits_\DD^\BK X^f$,
where $X^f$ denotes a functorial objectwise fibrant replacement of $X$ in $\M^\DD$. In other words \cite[XI.3, XI.8]{Bousfield_Kan}, there is a natural weak equivalence 
$\holim\nolimits_\DD X\wequiv\RR\holim\nolimits^\BK_\DD X$
and if furthermore, $X$ is objectwise fibrant, then $\holim\nolimits_\DD X\wequiv\holim\nolimits^\BK_\DD X$; here, we have denoted by $\RR\holim^\BK_\DD$ the total right derived functor of $\holim^\BK_\DD$.
\end{defn}

\section{The homotopy theory of $\K$-coalgebras}
\label{sec:homotopy_theory_K_coalgebras}

In this section we recall briefly the Arone-Ching enrichments and associated homotopy theory of $\K$-coalgebras \cite[Section 1]{Arone_Ching_classification} in the context needed for this paper. Compare also with the context in \cite{Ching_Harper_derived_Koszul_duality}, but note that here the Arone-Ching enrichments are considerably simpler since every object in $\Space_*$ is cofibrant and every object in $\sAb$ is fibrant in the underlying category $\Space_*$.

\begin{defn}
\label{defn:model_cat_coAlgK_language} A morphism in $\coAlgK$ is a \emph{cofibration} if the underlying morphism in $\sAb$ is a cofibration. An object $Y$ in $\coAlgK$ is \emph{cofibrant} if the unique map $\emptyset\rarrow Y$ in $\coAlgK$ is a cofibration.
\end{defn}

\begin{rem}
\label{rem:initial_object_K_coalgebras}
In $\coAlgK$ the initial object $\emptyset$ and the terminal object $*$ are isomorphic. Here, the terminal object is the trivial $\K$-coalgebra with underlying object $0$. This is because there is an adjunction
$
\xymatrix@1{
  \coAlgK\ar@<0.5ex>[r] &
  \sAb\!:\K\ar@<0.5ex>[l] 
}
$
with $\K$ right adjoint to the forgetful functor on top, together with the fact that right adjoints preserve terminal objects, and the calculation that $\K0=\tilde{\ZZ}(*)=0$.
\end{rem}

Recall that a morphism of $\K$-coalgebras from $Y$ to $Y'$ is a map $\function{f}{Y}{Y'}$ in $\sAb$ that respects the $\K$-coaction; i.e., such that $(\K f)m=mf$. This motivates the following cosimplicial resolution of $\K$-coalgebra maps from $Y$ to $Y'$.

\begin{defn}
\label{defn:derived_K_coalgebra_maps_second_attempt}
Let $Y,Y'$ be cofibrant $\K$-coalgebras. The cosimplicial object (in $\sSet$) $\Hombold_\sAb(Y,\K^\bullet Y')$  looks like (showing only the coface maps)
\begin{align}
\label{eq:derived_K_coalgebra_maps_second_attempt}
\xymatrix{
  \Hombold_\sAb(Y,Y')\ar@<0.5ex>[r]^-{d^0}\ar@<-0.5ex>[r]_-{d^1} &
  \Hombold_\sAb(Y,\K Y')
  \ar@<1.0ex>[r]\ar[r]\ar@<-1.0ex>[r] &
  \Hombold_{\sAb}(Y,\K\K Y')\cdots
}
\end{align}
and is defined objectwise by $\Hombold_\sAb(Y,\K^\bullet Y')^n:=\Hombold_\sAb(Y,\K^n Y')$ with the obvious coface and codegeneracy maps induced by the comultiplication and coaction maps, and counit map, respectively; see \cite[1.3]{Arone_Ching_classification}.
\end{defn}

Recall the usual notion of realization of a simplicial set, regarded as taking values in the category of compactly generated Hausdorff spaces, denoted $\CGHaus$ (e.g., \cite{Goerss_Jardine}).

\begin{defn}
\label{defn:realization_sSet}
The \emph{realization} functor $|-|$ for simplicial sets is defined objectwise by the coend
$\function{|-|}{\sSet}{\CGHaus}$,
$X\mapsto X \times_{\Delta}\Delta^{(-)}$.
Here, $\Delta^n$ in $\CGHaus$ denotes the topological standard $n$-simplex for each $n\geq 0$ (see Goerss-Jardine \cite[I.1.1]{Goerss_Jardine}).
\end{defn}

\begin{defn}
Let $X,Y$ be pointed spaces. The mapping space $\Map_{\Space_*}(X,Y)$ in $\CGHaus$ is defined by realization
$
  \Map_{\Space_*}(X,Y)
  :=|\Hombold_{\Space_*}(X,Y)|
$
of the indicated simplicial set.
\end{defn}

The following definition of the mapping space of derived $\K$-coalgebra maps appears in Arone-Ching \cite[1.10]{Arone_Ching_classification} and is a key ingredient in both the statements and proofs of our main results.

\begin{defn}
\label{defn:mapping_spaces_of_derived_K_coalgebras}
Let $Y,Y'$ be cofibrant $\K$-coalgebras. The \emph{mapping spaces} of derived $\K$-coalgebra maps $\Hombold_{\coAlgK}(Y,Y')$ in $\sSet$ and $\Map_{\coAlgK}(Y,Y')$ in $\CGHaus$ are defined by the restricted totalizations
\begin{align*}
  \Hombold_{\coAlgK}(Y,Y')
  &:=\Tot^\res\Hombold_\sAb\bigl(Y,\K^\bullet Y'\bigr)\\
  \Map_{\coAlgK}(Y,Y')
  &:=\Tot^\res\Map_\sAb\bigl(Y,\K^\bullet Y'\bigr)
\end{align*}
of the indicated cosimplicial objects.
\end{defn}

Recall the following useful propositions.

\begin{prop}
\label{prop:tot_commutes_with_realization}
If $Y\in(\sSet)^{\Delta_\res}$ and $Z\in(\sSet)^\Delta$ are objectwise fibrant, then the natural maps
$|\Tot^\res Y|\xrightarrow{\wequiv}\Tot^\res|Y|$ and 
$|\holim\nolimits^\BK_\Delta Z|\xrightarrow{\wequiv}
  \holim\nolimits^\BK_\Delta |Z|$
in $\CGHaus$ are weak equivalences.
\end{prop}

\begin{proof}
This is proved in \cite[6.15, 8.2]{Ching_Harper_derived_Koszul_duality}.
\end{proof}

The following corollary plays a key role in this paper.

\begin{prop}
Let $Y,Y'$ be cofibrant $\K$-coalgebras. Then the natural map of the form
$
  |\Hombold_{\coAlgK}(Y,Y')|\xrightarrow{\wequiv}
  \Map_{\coAlgK}(Y,Y')
$
is a weak equivalence.
\end{prop}

\begin{proof}
This follows from Proposition \ref{prop:tot_commutes_with_realization}.
\end{proof}

The following provides a useful language for working with the spaces of derived $\K$-coalgebra maps; see \cite[1.11]{Arone_Ching_classification}.

\begin{defn}
\label{defn:derived_K_coalgebra_map}
Let $Y,Y'$ be cofibrant $\K$-coalgebras. A \emph{derived $\K$-coalgebra map} $f$ of the form $Y\rarrow Y'$ is any map in $(\sSet)^{\Delta_\res}$ of the form
$
  \function{f}{\Delta[-]}{\Hombold_\sAb\bigl(Y,\K^\bullet Y'\bigr)}.
$
A \emph{topological derived $\K$-coalgebra map} $g$ of the form $Y\rarrow Y'$ is any map in $(\CGHaus)^{\Delta_\res}$ of the form
$
  \function{g}{\Delta^\bullet}{\Map_\sAb\bigl(Y,\K^\bullet Y'\bigr)}.
$
The \emph{underlying map} of a derived $\K$-coalgebra map $f$ is the map $\function{f_0}{Y}{Y'}$ that corresponds to the map $\function{f_0}{\Delta[0]}{\Hombold_\sAb(Y,Y')}$. Note that every derived $\K$-coalgebra map $f$ determines a topological derived $\K$-coalgebra map $|f|$ by realization.
\end{defn}

\begin{defn}
If $X,Y\in(\sSet)^\Delta$ their \emph{box product} $X\square Y\in(\sSet)^\Delta$ is defined objectwise by a coequalizer of the form
\begin{align*}
  (X\square Y)^n
  \Iso\colim\Bigl(
  \xymatrix{
    \coprod\limits_{p+q=n} X^{p}\times Y^{q} & 
    \coprod\limits_{r+s=n-1} X^{r}\times Y^{s}
    \ar@<-1.0ex>[l]\ar@<0.0ex>[l]
  }
\Bigr)
\end{align*}
where the top (resp. bottom) map is induced by $\id\times d^0$ (resp. $d^{r+1}\times\id$) on each $(r,s)$ term of the indicated coproduct; note that $(X\square Y)^0\Iso X^0\times Y^0$. The coface maps $\function{d^i}{(X\square Y)^n}{(X\square Y)^{n+1}}$ are induced by
\begin{align*}
  \left\{
  \begin{array}{lr}
  X^p\times Y^q\xrightarrow{d^i\times\id}X^{p+1}\times Y^q, & \text{if $i\leq p$,}\\
  X^p\times Y^q\xrightarrow{\id\times d^{i-p}}X^p\times Y^{q+1}, & \text{if $i>p$,}
  \end{array}
  \right.
\end{align*}
and the codegeneracy maps $\function{s^j}{(X\square Y)^n}{(X\square Y)^{n-1}}$ are induced by
\begin{align*}
  \left\{
  \begin{array}{lr}
  X^p\times Y^q\xrightarrow{s^j\times\id}X^{p-1}\times Y^q, & \text{if $j<p$,}\\
  X^p\times Y^q\xrightarrow{\id\times s^{j-p}}X^p\times Y^{q-1}, & \text{if $j\geq p$,}
  \end{array}
  \right.
\end{align*}
If $(\M,\tensor)$ is any closed symmetric monoidal category and $X,Y\in\M^\Delta$, then their box product $X\square Y\in\M^\Delta$ is defined similarly by replacing $(\sSet,\times)$ with $(\M,\tensor)$; for instance, with $(\CGHaus,\times)$.
\end{defn}

\begin{rem}
If $X,Y\in(\sSet)^\Delta$ their box product $X\square Y\in(\sSet)^\Delta$ is the left Kan extension of objectwise product along ordinal sum (or concatenation). This is proved in \cite[2.3]{McClure_Smith_cosimplicial_objects}; see also Batanin \cite[Section 2]{Batanin_coherent} and McClure-Smith \cite{McClure_Smith_solution}; a dual version of the construction appears in Artin-Mazur \cite[III]{Artin_Mazur}.
\end{rem}

\begin{prop}
\label{prop:composition_map}
Let $Y,Y',Y''$ be cofibrant $\K$-coalgebras. There is a natural map of the form
$
  \function{\mu}
  {\Hombold_\sAb\bigl(Y,\K^\bullet Y'\bigr)\square
  \Hombold_\sAb\bigl(Y',\K^\bullet Y''\bigr)}
  {\Hombold_\sAb\bigl(Y,\K^\bullet Y''\bigr)}
$
in $(\sSet)^\Delta$. We sometimes refer to $\mu$ as the \emph{composition} map.
\end{prop}

\begin{proof}
This is proved exactly as in \cite[1.6]{Arone_Ching_classification}; $\mu$ is the map induced by the collection of composites
\begin{align*}
  &\Hombold_\sAb\bigl(Y,\K^p Y'\bigr)\times
  \Hombold_\sAb\bigl(Y',\K^q Y''\bigr)\xrightarrow{\id\times\K^p}\\
  &\Hombold_\sAb\bigl(Y,\K^p Y'\bigr)\times
  \Hombold_\sAb\bigl(\K^p Y',\K^p \K^q Y''\bigr)\xrightarrow{\mathrm{comp}}
  \Hombold_\sAb\bigl(Y,\K^{p+q} Y''\bigr)
\end{align*}
where $p,q\geq 0$.
\end{proof}

\begin{prop}
\label{prop:realization_commutes_with_box_product}
Let $A,B\in(\sSet)^\Delta$. There is a natural isomorphism of the form
$
  |A\square B|\Iso|A|\square|B|
$
in $(\CGHaus)^\Delta$.
\end{prop}

\begin{proof}
This follows from the fact that realization commutes with finite products and all small colimits.
\end{proof}

\begin{prop}
\label{prop:composition_map_topological}
Let $Y,Y',Y''$ be cofibrant $\K$-coalgebras. There is a natural map of the form
$
  \function{\mu}{
  \Map_\sAb\bigl(Y,\K^\bullet Y'\bigr)\square
  \Map_\sAb\bigl(Y',\K^\bullet Y''\bigr)}
  {\Map_\sAb\bigl(Y,\K^\bullet Y''\bigr)}
$
in $(\CGHaus)^\Delta$. We sometimes refer to $\mu$ as the \emph{composition} map.
\end{prop}

\begin{proof}
This follows from Proposition \ref{prop:composition_map} by applying realization, together with Proposition \ref{prop:realization_commutes_with_box_product}.
\end{proof}

\begin{defn}
Let $Y$ be a cofibrant $\K$-coalgebra. The \emph{unit map} $\iota$ is the map
$*\rightarrow\Map_{\sAb}(Y,\K^\bullet Y)$
in $(\CGHaus)^\Delta$ which is realization of the coaugmentation map \cite[1.6]{Arone_Ching_classification} of the form $*\rarrow\Hombold_{\sAb}(Y,\K^\bullet Y)$ that picks out the identity map on $Y$ in simplicial degree 0.
\end{defn}

\begin{defn}
The non-$\Sigma$ operad $\mathsf{A}$ in $\CGHaus$ is the coendomorphism operad of $\Delta^\bullet$ with respect to the box product $\square$ (\cite[1.12]{Arone_Ching_classification}) and is defined objectwise by the end construction
$
  \mathsf{A}(n):=
  \Map_{\Delta_\res}\bigl(\Delta^\bullet,(\Delta^\bullet)^{\square n}\bigr)
  :=\Map\bigl(\Delta^\bullet,(\Delta^\bullet)^{\square n}\bigr)^{\Delta_\res}
$.
In other words, $\mathsf{A}(n)$ is the space of restricted cosimplicial maps from $\Delta^\bullet$ to $(\Delta^\bullet)^{\square n}$; in particular, note that $\mathsf{A}(0)=*$.
\end{defn}

Consider the natural collection \cite[1.13]{Arone_Ching_classification} of maps of the form ($n\geq 0$)
\begin{align}
\label{eq:A_infinity_composition_maps}
  \mathsf{A}(n)\times
  \Map_\coAlgK(Y_0,Y_1)\times\cdots\times
  \Map_\coAlgK(Y_{n-1},Y_n)
  \rightarrow
  \Map_\coAlgK(Y_0,Y_n)
\end{align}
induced by (iterations of) the composition map $\mu$ (Proposition \ref{prop:composition_map_topological}); in particular, in the case $n=0$, note that \eqref{eq:A_infinity_composition_maps} denotes the map
$
  *=\mathsf{A}(0)\rightarrow\Map_\coAlgK(Y_0,Y_0)
$
that is $\Tot^\res$ applied to the unit map.

\begin{prop}
\label{prop:A_infinity_enrichment}
The collection of maps \eqref{eq:A_infinity_composition_maps} determines a topological $A_\infty$ category with objects the cofibrant $\K$-coalgebras and morphism spaces the mapping spaces $\Map_\coAlgK(Y,Y')$.
\end{prop}

\begin{proof}
This is proved exactly as in \cite[1.14]{Arone_Ching_classification}.
\end{proof}

\begin{defn}
The \emph{homotopy category} of $\K$-coalgebras (see \cite[1.15]{Arone_Ching_classification}), denoted $\Ho(\coAlgK)$, is the category with objects the cofibrant $\K$-coalgebras and morphism sets $[Y,Y']_\K$ from $Y$ to $Y'$ the path components
$
  [Y,Y']_\K := \pi_0\Map_\coAlgK(Y,Y')
$
of the indicated mapping spaces.
\end{defn}

\begin{prop}
Let $Y,Y'$ be cofibrant $\K$-coalgebras. There is a natural map of morphism sets of the form
$\function{\pi}{\hom_\coAlgK(Y,Y')}{[Y,Y']_\K}$.
\end{prop}

\begin{prop}
There is a well-defined functor
$
  \function{\gamma}{\coAlg^\mathrm{c}_K}{\Ho(\coAlgK)}
$
that is the identity on objects and is the map $\pi$ on morphisms; here, $\coAlg^\mathrm{c}_K\subset\coAlgK$ denotes the full subcategory of cofibrant $\K$-coalgebras.
\end{prop}

\begin{proof}
This is proved exactly as in \cite[1.14]{Arone_Ching_classification}.
\end{proof}

\begin{defn}
\label{defn:weak_equivalence_of_K_coalgebras}
A derived $\K$-coalgebra map $f$ of the form $Y\rarrow Y'$ is a \emph{weak equivalence} if the underlying map $\function{f_0}{Y}{Y'}$ is a weak equivalence.
\end{defn}

\begin{prop}
\label{prop:weak_equivalence_if_and_only_if_iso_in_homotopy_category}
Let $Y,Y'$ be cofibrant $\K$-coalgebras. A derived $\K$-coalgebra map $f$ of the form $Y\rarrow Y'$ is a weak equivalence if and only if the induced map $\gamma(f)$ in $[Y,Y']_\K$ is an isomorphism in the homotopy category of $\K$-coalgebras.
\end{prop}

\begin{proof}
The is proved exactly as in \cite[1.16]{Arone_Ching_classification}.
\end{proof}

\section{The derived adjunction}
\label{sec:derived_fundamental_adjunction}

The derived unit is the map of pointed spaces of the form $X\rarrow\holim_\Delta C(\tilde{\ZZ}X)$ corresponding to the identity map $\function{\id}{\tilde{\ZZ} X}{\tilde{\ZZ} X}$; it is tautologically the Bousfield-Kan $\ZZ$-completion map $X\rarrow X^\wedge_\ZZ$ in \cite[I.4]{Bousfield_Kan}. The derived counit is the derived $\K$-coalgebra map of the form $\tilde{\ZZ}\holim_\Delta C(Y)\rightarrow Y$ corresponding to the identity map $\function{\id}{\holim_\Delta C(Y)}{\holim_\Delta C(Y)}$, after taking into account the natural zigzags of weak equivalences
$
  \holim\nolimits_\Delta C(Y)
  \wequiv\Tot^\res C(Y)
$
of pointed spaces.

\begin{defn}
\label{defn:derived_counit_map}
The \emph{derived counit map} associated to \eqref{eq:derived_adjunction_main} is the derived $\K$-coalgebra map of the form $\tilde{\ZZ}\holim_\Delta C(Y)\rightarrow Y$, with 
$\tilde{\ZZ}\Tot^\res C(Y)\rightarrow Y$ the underlying map corresponding to the identity map
$\function{\id}{\Tot^\res C(Y)}{\Tot^\res C(Y)}$
in $\Space_*$, via the adjunctions \eqref{eqref:tot_adjunctions} and \eqref{eq:homology_adjunction}. In more detail, the derived counit map is the derived $\K$-coalgebra map defined by the composite (see \cite[2.17]{Arone_Ching_classification})
\begin{align}
\label{eq:derived_counit_map}
  \Delta[-]\xrightarrow{(*)}
  \Hombold_{\Space_*}\bigl(\Tot^\res C(Y),C(Y)\bigr)
  \Iso
  \Hombold_\sAb\bigl(\tilde{\ZZ}\Tot^\res C(Y),\K^\bullet Y\bigr)
\end{align}
in $(\sSet)^{\Delta_\res}$, where $(*)$ corresponds to the identity map on $\Tot^\res C(Y)$ in $\Space_*$, via the adjunctions \eqref{eqref:tot_adjunctions} and \eqref{eq:homology_adjunction}. 
\end{defn}

\begin{prop}
\label{prop:natural_tot_map_induced_by_simplicial_functor}
Let $\M,\M'$ be simplicial model categories. Let $\function{F}{\M}{\M'}$ be a simplicial functor and $X$ a cosimplicial (resp. restricted cosimplicial) object in $\M$. There are maps of the form
$F\Tot(X)\rarrow\Tot(FX)$ and $F\Tot^\res(X)\rarrow\Tot^\res(FX)$
(in $\M'$) induced by the simplicial structure maps of $F$.
\end{prop}

\begin{proof}
In both cases, the indicated map is induced by the composite maps
\begin{align*}
  F\bigl(\hombold(\Delta[n],X^n)\bigr)\tensordot\Delta[n]
  \xrightarrow{\sigma}
  F\bigl(\hombold(\Delta[n],X^n)\tensordot\Delta[n]\bigr)
  \xrightarrow{\id(\ev)}
  F(X^n),\quad\quad n\geq 0,
\end{align*}
via the natural isomorphisms in Remark \ref{rem:useful_adjunction_isomorphisms_simplicial_structure}.
\end{proof}

Consider the collection of maps
$\Delta[n]
  \rightarrow
  \Hombold_\sAb\bigl(\tilde{\ZZ}\Tot^\res C(Y),\K^n Y\bigr)$, 
$n\geq 0$,
in $\sSet$ described in \eqref{eq:derived_counit_map} associated to the derived counit map. It follows from the adjunction isomorphisms that these maps correspond with the maps
\begin{align}
\label{eq:canonical_maps_for_derived_counit}
  \bigl(\tilde{\ZZ}\Tot^\res C(Y)\bigr)\tensordot\Delta[n]
  \rightarrow \K^n Y,
  \quad\quad
  n\geq 0,
\end{align}
in $\sAb$, defined by the composite
\begin{align*}
  \bigl(\tilde{\ZZ}\Tot^\res C(Y)\bigr)\tensordot\Delta[n]
  \rightarrow
  \bigl(\Tot^\res \tilde{\ZZ} C(Y)\bigr)\tensordot\Delta[n]
  \xrightarrow{\mathrm{(*)}}
  \K(\K)^n Y\xrightarrow{\varepsilon(\id)^n\id=s^{-1}}\id(\K)^n Y
\end{align*}
where $(*)$ denotes the indicated projection map; here, it may be helpful to note that $\tilde{\ZZ}C(Y)=\Cobar(\K,\K,Y)$.

\begin{prop}
\label{prop:induced_map_on_mapping_spaces_built_from_Q}
Let $X,X'$ be pointed spaces. There are natural morphisms of mapping spaces of the form
$\tilde{\ZZ}:\Map_{\Space_*}(X,X')
  \rarrow\Map_\coAlgK(\tilde{\ZZ}X,\tilde{\ZZ}X')$
in $\CGHaus$.
\end{prop}

\begin{proof}
Consider the composite
\begin{align*}
  \Hombold_{\Space_*}(X,X')&\rightarrow
  \Tot^\res\Hombold_{\Space_*}(X,X')\\
  &\xrightarrow{(*)}
  \Tot^\res\Hombold_{\Space_*}\bigl(X,C(\tilde{\ZZ}X')\bigr)\Iso
  \Hombold_\coAlgK(\tilde{\ZZ}X,\tilde{\ZZ}X')
\end{align*}
The proposition follows by applying realization and using Proposition \ref{prop:tot_commutes_with_realization}; here, the map $(*)$ is induced by the natural coaugmentation $X'\rarrow C(\tilde{\ZZ}X')$ in $(\Space_*)^\Delta$.
\end{proof}

\begin{prop}
There is an induced functor
$\function{\tilde{\ZZ}}{\Ho(\Space_*)}{\Ho(\coAlgK)}$
which on objects is the map $X\mapsto \tilde{\ZZ}X$ and on morphisms is the map
$
  [X,X']\rightarrow[\tilde{\ZZ}X,\tilde{\ZZ}X']_K
$
which sends $[f]$ to $[\tilde{\ZZ}(f)]$ obtained by taking path components.
\end{prop}

\begin{proof}
This follows from Proposition \ref{prop:induced_map_on_mapping_spaces_built_from_Q} (see \cite[2.20]{Arone_Ching_classification}).
\end{proof}

The following three propositions, which are exercises left to the reader, verify that the cosimplicial resolutions of $\K$-coalgebra mapping spaces respect the adjunction isomorphisms associated to the $(\tilde{\ZZ},U)$ adjunction (Proposition \ref{prop:cosimplicial_resolutions_of_K_coalgebras_respect_adjunction_isos}).

\begin{prop}
\label{prop:adjunction_isos_respect_cosimplicial_relations}
Let $X\in\Space_*$ and $Y\in\coAlgK$. The adjunction isomorphisms associated to the $(\tilde{\ZZ},U)$ adjunction induce well-defined isomorphisms of $\Delta$-shaped diagrams
$
  \hom_{\sAb}(\tilde{\ZZ}X,\K^\bullet Y)\Iso
  \hom_{\Space_*}(X,U\K^\bullet Y)
$
in $\Set$, natural in $X,Y$.
\end{prop}

\begin{prop}
\label{prop:sigma_is_a_map_of_K_coalgebras_general_statement}
If $Y\in\coAlgK$ with comultiplication map $\function{m}{Y}{\K Y}$ and $L\in\sSet$, then $Y\tensordot L$ in $\sAb$ has a natural $\K$-coalgebra structure with  comultiplication map $\function{m}{Y\tensordot L}{\K(Y\tensordot L)}$ given by the composite
$
  Y\tensordot L\xrightarrow{m\tensordot\id}
  \K(Y)\tensordot L\xrightarrow{\sigma}
  \K(Y\tensordot L)
$.
\end{prop}

\begin{prop}
\label{prop:natural_isomorphism_sigma_respects_cosimplicial_relations}
Let $X\in\Space_*$ and $Y\in\coAlgK$. Then $\function{\sigma}{\tilde{\ZZ}(X)\tensordot L}{\tilde{\ZZ}(X\tensordot L)}$ induces well-defined isomorphisms
$
  \hom_{\sAb}\bigl(\tilde{\ZZ}(X\tensordot L),\K^\bullet Y\bigr)
  \Iso
  \hom_{\sAb}\bigl(\tilde{\ZZ}(X)\tensordot L,\K^\bullet Y\bigr)
$
of $\Delta$-shaped diagrams in $\Set$, natural in $X,Y$.
\end{prop}

\begin{prop}
\label{prop:cosimplicial_resolutions_of_K_coalgebras_respect_adjunction_isos}
Let $X\in\Space_*$ and $Y\in\coAlgK$. The adjunction isomorphisms associated to the $(\tilde{\ZZ},U)$ adjunction induce well-defined isomorphisms of $\Delta$-shaped diagrams
$
  \Hombold_{\sAb}(\tilde{\ZZ}X,\K^\bullet Y)\Iso
  \Hombold_{\Space_*}(X,U\K^\bullet Y)
$
 in $\sSet$, natural in $X,Y$.
\end{prop}

\begin{proof}
It suffices to verify that the composite
\begin{align*}
  \hom(\tilde{\ZZ}(X)\tensordot\Delta[n],\K^\bullet Y)\Iso
  \hom(\tilde{\ZZ}(X\tensordot\Delta[n]),\K^\bullet Y)\Iso
  \hom(X\tensordot\Delta[n],U\K^\bullet Y)
\end{align*}
is a well-defined map of cosimplicial objects in $\Set$, natural in $X,Y$, for each $n\geq 0$; this follows from Propositions \ref{prop:adjunction_isos_respect_cosimplicial_relations} and \ref{prop:natural_isomorphism_sigma_respects_cosimplicial_relations}.
\end{proof}

\begin{prop}
\label{prop:zigzag_of_weak_equivalences_tot_and_tq_completion}
If $X$ is a pointed space, then there is a zigzag of weak equivalences of the form
$
  X^\wedge_\ZZ\wequiv\holim\nolimits_{\Delta}C(\tilde{\ZZ}X)\wequiv\Tot^\res C(\tilde{\ZZ}X)
$
in $\Space_*$, natural with respect to all such $X$.
\end{prop}

\begin{defn}
A pointed space $X$ is \emph{$\ZZ$-complete} if the natural coaugmentation $X\wequiv X^\wedge_\ZZ$ is a weak equivalence.
\end{defn}

\begin{prop}
\label{prop:fundamental_adjunction_derived_version}
There are natural zigzags of weak equivalences in $\CGHaus$ of the form
$
  \Map_\coAlgK(\tilde{\ZZ}X,Y)\wequiv
  \Map_{\Space_*}(X,\holim\nolimits_\Delta C(Y))
$ and applying $\pi_0$ gives the natural isomorphism $[\tilde{\ZZ}X,Y]_\K\Iso[X,\holim_\Delta C(Y)]$.
\end{prop}

\begin{proof}
There are natural zigzags of weak equivalences of the form (see \cite[2.20]{Arone_Ching_classification})
\begin{align*}
  \Hombold_{\Space_*}(X,\holim\nolimits_\Delta C(Y))
  &\wequiv\Hombold_{\Space_*}\bigl(X,\Tot^\res C(Y)\bigr)
  \Iso\Tot^\res\Hombold_{\Space_*}\bigl(X,U\K^\bullet Y\bigr)\\
  &\Iso\Tot^\res\Hombold_{\sAb}\bigl(\tilde{\ZZ}X,\K^\bullet Y\bigr)
  \Equal\Hombold_\coAlgK(\tilde{\ZZ}X,Y)
\end{align*}
in $\sSet$; applying realization, together with Proposition \ref{prop:tot_commutes_with_realization} finishes the proof.
\end{proof}

The following amounts to the observation that mapping into fibrant $\ZZ$-complete objects induces the indicated weak equivalence on mapping spaces; compare \cite[5.5]{Hess} and \cite[2.15]{Arone_Ching_classification}. It shows that the integral chains functor in \eqref{eq:derived_adjunction_main} is homotopically fully faithful on $\ZZ$-complete spaces.

\begin{prop}
\label{prop:formal_adjunction_and_iso_argument}
Let $X,X'$ be pointed spaces. If $X'$ is $\ZZ$-complete and fibrant, then there is a natural zigzag
$
  \tilde{\ZZ}\colon\thinspace\Map_{\Space_*}(X,X')\wequiv\Map_\coAlgK(\tilde{\ZZ}X,\tilde{\ZZ}X')
$
of weak equivalences; applying $\pi_0$ gives the map $[f]\mapsto[\tilde{\ZZ}(f)]$.
\end{prop}

\begin{proof}
This follows from the natural zigzags
\begin{align*}
  \Map_{\Space_*}(X,{X'}^\wedge_\ZZ)
  \wequiv\Map_{\Space_*}(X,\holim\nolimits_\Delta C(\tilde{\ZZ}X')\wequiv\Map_\coAlgK(\tilde{\ZZ}X,\tilde{\ZZ}X')
\end{align*}
of weak equivalences.
\end{proof}

\bibliographystyle{plain}
\bibliography{IntegralChains}

\end{document}